\documentclass[a4paper,10pt]{amsart}

\usepackage[latin1]{inputenc}
\usepackage{amssymb,amsmath,tikz}
\usepackage{enumitem}

\usepackage[all]{xy}
\usepackage{thmtools}
\usepackage{lineno} 

   \usepackage[pagebackref]{hyperref}
  \hypersetup{
colorlinks=true,
linkcolor=cyan,
citecolor=cyan
}
 
\usepackage{mathtools}

\numberwithin{equation}{section}

\declaretheorem[style=plain,parent=section]{theorem}
\declaretheorem[style=plain,sibling=theorem]{corollary}
\declaretheorem[style=plain,sibling=theorem]{lemma}

\declaretheorem[style=plain,sibling=theorem]{proposition}

\declaretheorem[style=definition,sibling=theorem]{definition}

\declaretheorem[style=definition, qed=\hfill $\diamond$, sibling=definition]{example}
\declaretheorem[style=remark,sibling=theorem]{remark}

\makeatletter
\@namedef{subjclassname@2020}{\textup{2020} Mathematics Subject Classification}
\makeatother

\usepackage{xcolor}
\newcommand{\red}[1]{\textcolor{red}{#1}}

\newcommand{\blue}[1]{\textcolor{blue}{#1}}

\title{An introduction to local tropicalization}
\author[P. Popescu-Pampu, D. Stepanov]{Patrick Popescu-Pampu and Dmitry Stepanov} 
\date{14 February 2025}

\subjclass[2020]{14B05 (primary), 14T10, 14T90, 32S25}
\keywords{Formal germs, local tropicalization, Newton polyhedra, semivaluations, splice type singularities, toric birational morphisms.}

\begin{document}

{\bf This text will appear in Volume 8 of the {\em Handbook of Geometry and Topology of Singularities}, 
Springer.}

\begin{abstract}
    In this paper we explain four viewpoints on the {\em local tropicalization} 
     of formal subgerms of toric germs, which is a local analog of the global tropicalization 
     of subvarieties of algebraic tori. We start by illustrating some of those viewpoints for plane 
     curve singularities, then we pass to arbitrary dimensions. We conclude by describing several  
     variants and extensions of the notion of local tropicalization presented in this paper.
\end{abstract}

\vspace{1cm}

\maketitle

\tableofcontents

\medskip
\section{Introduction}  \label{sect:intro}
\medskip

The aim of this paper is to present various viewpoints on the notion of {\em local tropicalization},  
introduced by the authors in their 2013 paper \cite{PS 13} as a local analog of the global notion of tropicalization of a subvariety of an algebraic torus. We assume that the reader is familiar with the basics of singularity theory and toric geometry. For details about global tropicalization, one may consult either the introductory papers \cite{K 08}, \cite{T 08}, \cite{K 09}  of Kajiwara, Teissier and Katz 
or Maclagan and Sturmfels' textbook \cite{MS 15}.

Let us explain briefly what does it mean to tropicalize locally a reduced 
complex formal germ $Y \hookrightarrow (\mathbb{C}^n, 0)$ 
whose irreducible components are not contained in the union $\partial \,  \mathbb{C}^n$ of coordinate hyperplanes of $\mathbb{C}^n$. The initial definitions of \cite[Definitions 6.6, 6.7]{PS 13} involved spaces of valuations, which allowed us to apply them both to local and global situations, for arbitrary commutative 
rings, not necessarily containing a field. 

But the shortest definition of the local tropicalization of the embedding 
\index{Local tropicalization}   \index{Tropicalization!local}
$Y \hookrightarrow (\mathbb{C}^n, 0)$ is most likely as {\em the topological closure inside $\mathbb{R}^n$ of the union of rays $\mathbb{R}_{\geq 0} w$ generated by the weight vectors $w=(w_1, \dots, w_n)$ which appear as initial exponent vectors of the formal arcs 
   \[  t \mapsto (c_1 t^{w_1} + \cdots \  , \  \dots \  , \ c_n t^{w_n}+ \cdots )  \]
 contained in $Y$ but not contained in $\partial \, \mathbb{C}^n$} (that is, such that $c_j \in \mathbb{C}^*$ and $w_j \in \mathbb{Z}_{>0}$ for every $j \in \{1, \dots, n\}$). For instance, if $Y$ is a plane curve singularity defined by 
 a series $f(x,y) \in \mathbb{C}[[x,y]]$, then its local tropicalization is the union of the rays orthogonal to the 
 compact edges of the Newton polygon of $f$ (see Corollary \ref{cor:arcweightplanecurve}).

For the geometrically-inclined reader, perhaps the most suggestive viewpoint on local tropicalization is the
following {\em toric-geometrical} one. Seen as an affine toric variety, $\mathbb{C}^n$  has the set
$\mathcal{F}_0$ of faces of the non-negative 
orthant $\sigma_0:= (\mathbb{R}_{\geq 0})^n$ of its vector space $\mathbb{R}^n$ of weight vectors 
as associated fan. Note that $\sigma_0$ also belongs to $\mathcal{F}_0$, as its nonproper face. 
Let $\mathcal{F}$ be a fan whose cones are contained in $\sigma_0$. There exists a 
canonical birational morphism 
    \[  \pi_{\mathcal{F}} : \mathcal{X}_{\mathcal{F}} \to \mathbb{C}^n \]
 from the toric variety $\mathcal{X}_{\mathcal{F}}$ associated to the fan $\mathcal{F}$ to the affine toric variety $\mathbb{C}^n = \mathcal{X}_{\mathcal{F}_0}$. This birational morphism is a {\em modification}, that is, it is moreover  {\em proper},  if and only if the {\em support} $| \mathcal{F} |$ of $\mathcal{F}$ (that is, if the union of its cones) is equal to $\sigma_0$. Denote by $Y_{\mathcal{F}}$ the strict transform of $Y$ by $\pi_{\mathcal{F}}$ and let 
      \[  \pi : Y_{\mathcal{F}} \to Y \]
be the restriction of $\pi_{\mathcal{F}}$ to $Y_{\mathcal{F}}$. {\em It may happen that $\pi$ is a modification without $ \pi_{\mathcal{F}}$ being so} (see Example \ref{ex:compactfan}). 

\medskip
    {\bf Question:} {\em Under which condition is the morphism $\pi : Y_{\mathcal{F}} \to Y $ proper?} 
\medskip

{\bf Answer:} {\em Precisely when the  support $| \mathcal{F} |$ of $\mathcal{F}$  
contains the local tropicalization of $Y \hookrightarrow (\mathbb{C}^n, 0)$}! That is, this local tropicalization may be defined as {\em the intersection of the supports of the fans $\mathcal{F}$ such that $\pi : Y_{\mathcal{F}} \to Y$ is proper}. 
\medskip

This is the announced toric-geometrical interpretation of the notion of local tropicalization. 
In this paper we will present two more definitions of local tropicalization for embeddings $Y \hookrightarrow (\mathbb{C}^n, 0)$ as above and we will explain the equivalences of the four definitions (see Theorem \ref{thm:coincidthm}). We will indicate also related works (see Section \ref{sect:variants}). 

\medskip
Let us describe the structure of the paper. 
Each time a notation is introduced, it is included in a $\boxed{\mathrm{box}}$ 
and each time a terminology is presented, it is written in  {\bf boldface characters}. 
We start by explaining basic ideas about local tropicalization in dimension two, 
for plane curve singularities. Namely, in Section \ref{sect:motivex} we illustrate the above 
toric interpretation of local tropicalization using a cuspidal cubic curve in the affine plane $\mathbb{C}^2$. 
In Section \ref{sect:Newtonpolyg} we recall the notion of {\em Newton polygon} of a plane 
curve singularity $Y \hookrightarrow (\mathbb{C}^2, 0)$ in a way suitable for our purposes 
and we explain in this context the equivalence of the above two definitions of local tropicalization. 
Then we pass to arbitrary dimensions. In Section \ref{sect:notations} we explain our notations 
and basic facts about {\em toric geometry} and {\em valuation theory}. 
In Section \ref{sect:weightsofarcsandvals} 
we turn the toric-geometric viewpoint on local tropicalization into a definition of local tropicalization 
for certain formal subgerms $Y$ of a germ of affine toric variety $(\mathcal{X}_{\sigma}, 0)$ and we explain how to associate weight vectors to certain arcs included in $Y$ or to semivaluations 
of the local ring of $Y$. In Section \ref{sect:coincid} we formulate  Theorem \ref{thm:coincidthm},  
stating that various definitions of the local tropicalization of an embedding of $Y$ 
into a toric germ  $(\mathcal{X}_{\sigma}, 0)$ are equivalent and we explain intuitively 
some ingredients of its proof. 
In Section \ref{sect:localtropprinc} we describe the local tropicalization of a principal 
effective divisor on a toric germ in terms of its Newton polyhedron. In Section \ref{sec:splicetype} 
we describe our work done in collaboration with Cueto on the
local tropicalizations of a large and important class of surface singularities, 
Neumann and Wahl's {\em splice type singularities}. Finally, in Section \ref{sect:variants} 
we present variants and extensions of the notion of local tropicalization explained in this paper, 
mentioning papers of Kajiwara, Payne, Ulirsch, Esterov, J. Giansiracusa, N. Giansiracusa and Lorscheid.

\medskip
\section{Motivating considerations about the cusp singularity}  \label{sect:motivex}
\medskip

In this section we illustrate the toric-geometrical interpretation of local tropicalization explained in Section \ref{sect:intro} on the example of a plane cubic curve with a cusp singularity at the origin. 
For this purpose, we recall basic facts about toric surfaces, their fans and their orbits (see points 
\eqref{inclpresbij}--\eqref{weightvect1paramgp} below).
\medskip

Throughout the paper, we denote by $\boxed{Z(f)}$ the principal Cartier divisor defined by a rational function $f$ on a normal complex analytic variety, or the germ of such a divisor defined by an element of a local ring of such a variety. We denote by $\boxed{\mathbb{C}^2_{x,y}}$ the complex affine plane with 
coordinates $x$ and $y$.

\begin{definition}  \label{def:modif}
      A {\bf modification}\index{Modification} of an irreducible complex analytic 
      variety $X$ is a proper bimeromorphic morphism $\pi : X_{\pi} \to X$. 
\end{definition}

An {\bf embedded resolution}\index{Resolution!embedded} of a curve contained in a smooth surface $X$ is a
modification $\pi$ of the surface such that $X_{\pi}$ is smooth, the total transform of the curve is a normal
crossings divisor and  $\pi$ induces an isomorphism between the complement of the exceptional set of $\pi$ in
$X_\pi$ and the complement of the singular locus of the curve in $X$.  There exists a {\bf minimal embedded
resolution}\index{Resolution!minimal}, through which all other resolutions factor. It is obtained by blowing up successively the points of the curve and of its total transforms where one does not have locally normal
crossings (see \cite[Section 3.4]{W 04}).

    In order to illustrate this fact, let $\boxed{Y} \hookrightarrow  \mathbb{C}^2_{x,y} = : \boxed{S_0}$  
be the plane cubic curve $Z(f)$, where:
     \begin{equation}  \label{eq:defcubic}
        f(x,y) := y^2 - 2 x^3 + x^2 y. 
      \end{equation}
 It has only one singular point, at the origin $\boxed{p_0}: = (0,0)$ of $S_0$, as may be seen by solving the system of equations  $f = \partial_x f = \partial_y f =0$. The germ of $Y$ at $p_0$ is 
 a {\bf cusp singularity}\index{Singularity!cusp}. That is, it is formally isomorphic 
 to the germ at the origin of the semicubical 
 parabola $Z(u^2 - v^3) \hookrightarrow  \mathbb{C}^2_{u,v}$. This fact may be shown explicitly 
 by first completing the square in the polynomial \eqref{eq:defcubic}:
   \[  f(x,y) = \left(y + \frac{x^2}{2}\right)^2 - 2x^3 \left(1 + \frac{x}{8}\right) \]
 and by performing then the formal (even analytic) change of variables:
   \[ u := y + \frac{x^2}{2}, \  \  \  \   v := 2^{1/3} x \left(1 + \frac{x}{8}\right)^{1/3}, \]
  in which the last cubic root is expanded using Newton's binomial series expansion formula for 
  $(1 + t)^{1/3}$. 
  
 One may reach the minimal embedded resolution of $Y$ by blowing up successively three points 
  $p_0, p_1, p_2$ of $Y$ and of its strict transforms, as illustrated on the upper part of Figure \ref{fig:embrescusp}. 
  For each $i \in \{0, 1, 2\}$, denote by
    \[ \boxed{\pi_i} : S_{i+1} \to  S_i \]
  the blowup morphism of the surface $\boxed{S_i}$ at the point $p_i$ and by 
  $\boxed{E_i} \hookrightarrow S_{i+1}$ the exceptional divisor of $\pi_i$. It is a smooth 
  rational curve of self-intersection $-1$. For each pair $(i,j)$ such that $j \geq i +2$, 
  we denote by $\boxed{E_{i,j}}$  the strict transform of $E_i$ on the surface $S_j$ and by $\boxed{Y_j}$ the strict transform of the curve $Y$ on the same surface. For instance, $Y_0 = Y$.  
  
  Each time one blows up a smooth point of a compact curve contained in a smooth surface, 
  the self-intersection number of its strict transform is one less than the initial self-intersection number 
  (see \cite[Lemma 4.3]{L 71} or \cite[Lemma 8.1.6]{W 04}). 
  Therefore, one has the following self-intersection numbers on the final surface $S_3$:
      \[ E_{0,3}^2 =  -3, \  E_{1,3}^2 =  -2, \  E_2^2 =  -1.   \]
   As $E_{0,3}$ and $E_{1,3}$ have negative self-intersections, they may be contracted to 
   normal singular points (see \cite[Theorem 4.9]{L 71}). 
   Denote by $\boxed{\tilde{S}}$ the resulting surface and by 
      \[ \boxed{\gamma} : S_3 \to \tilde{S} \]
   the contraction morphism (see the right diagonal arrow of Figure \ref{fig:embrescusp}).

     \begin{figure}[h!]
    \begin{center}
\begin{tikzpicture}[scale=0.48]

\begin{scope}[shift={(1,0)}]
        \draw [-, color=green!70!blue, thick] (-1,0)--(1,0);
        \draw [-, color=blue, thick] (0,-1)--(0,1);
 \draw[very thick, color=orange](0,0) .. controls (1,0.2) ..(1.7,1.5); 
\draw[very thick, color=orange](0,0) .. controls (1,-0.2) ..(1.7,-1.5); 
\node [above, color=orange] at (1.9,1.7) {$Y$};

\node[draw,circle, inner sep=1.5pt,color=black, fill=black] at (0,0){};
\node [below, color=black] at (-0.4,0) {$p_0$};
\node [above, color=black] at (-1,2) {$\boxed{S_0}$};
\node [below, color=blue] at (0, -1) {$Z(x)$};
\node [left, color=green!70!blue] at (-1, 0) {$Z(y)$};
\end{scope}


\begin{scope}[shift={(6,0)}]
     \draw [-, color=black, thick](-2,0) -- (2,0);
     \draw [-, color=green!70!blue, thick] (0,-1)--(0,1);
          \node [below, color=green!70!blue] at (0.2, -1) {$Z(y)$};
      \draw [-, color=blue, thick] (-1.2,-1)--(-1.2,1);
        \node [below, color=blue] at (-1.2, -1) {$Z(x)$};
\draw[very thick, color=orange] (-1.3,1.5) .. controls (0,-0.5) .. (1.3,1.5);
\node [above, color=orange] at (1.5, 1.7) {$Y_1$};

\node [below, color=black] at (-0.5,0) {$p_1$};
\node [below, color=black] at (2,0) {$E_0$};
\node[draw,circle, inner sep=1.5pt,color=black, fill=black] at (0,0){};
\node [above, color=black] at (-0.5,2) {$\boxed{S_1}$};

\end{scope}


\begin{scope}[shift={(12,0)}]
     \draw [-, color=black, thick](-2,0) -- (2,0);
      \draw [-, color=green!70!blue, thick] (-1.2,-1)--(-1.2,1);
            \node [above, color=green!70!blue] at (-1, 1) {$Z(y)$};
       \draw [-, color=blue, thick] (-1,-1.4)--(1,-1.4);
         \node [right, color=blue] at (1, -1.4) {$Z(x)$};
\draw [-, color=black, thick] (0,-2)--(0,2);
\node [right, color=black] at (0,-2.2) {$E_{0,2}$};
\node [above, color=black] at (2,-0.9) {$E_1$};
\draw[very thick, color=orange](-1.7,-1.7) -- (1.7, 1.7); 
\node [above, color=orange] at (1.9,1.7) {$Y_2$};

\node[draw,circle, inner sep=1.5pt,color=red, fill=red] at (0,0){};
\node [below, color=red] at (0.5,0) {$p_2$};
\node [above, color=black] at (-0.7,2) {$\boxed{S_2}$};
\end{scope}


\begin{scope}[shift={(19,0)}]
      \draw [-, color=red, very thick](-2.5,0) -- (2.5,0);
      \draw [-, color=blue, thick] (-2.5,-1)--(-0.5,-1);
          \node [left, color=blue] at (-2.5, -1) {$Z(x)$};
      \draw [-, color=green!70!blue, thick] (0.5,-1)--(2.5,-1);
            \node [right, color=green!70!blue] at (2.5, -1) {$Z(y)$};
\draw [-, color=orange, very thick] (0,-1.5)--(0,1.5);
\node [above, color=orange] at (0,1.7) {$Y_3$};
\node [above, color=red] at (-2.5,0) {$E_2$};
\node [right, color=black] at (-1.6,-1.8) {$E_{0,3}$};
\node [right, color=black] at (1.4,-1.8) {$E_{1,3}$};
\node[draw,circle, inner sep=1.5pt,color=black, fill=black] at (0,0){};
\node [below, color=black] at (0.5,0) {$p_3$};
\draw[thick, color=black](1.5,-1.5) -- (1.5, 1.5); 
\draw[thick, color=black](-1.5,-1.5) -- (-1.5, 1.5); 
    \node [above, color=black] at (-1.7,2) {$\boxed{S_3}$};
\end{scope}


\begin{scope}[shift={(9.5,-6)}]
     \draw [-, color=red, very thick](-2.5,0) -- (2.5,0);
\draw [-, color=orange, very thick] (0,-1.5)--(0,1.5);
\node [above, color=orange] at (0,1.7) {$\tilde{Y}$};
\node [above, color=red] at (-2.5,0) {$\tilde{E}_2$};
\node[draw,circle, inner sep=1.5pt,color=black, fill=black] at (0,0){};
\node [below, color=black] at (0.5,0) {$p_3$};
\draw[thick, color=green!70!blue](1.5,-1) -- (1.5, 1); 
      \node [below, color=green!70!blue] at (1.5, -1) {$Z(y)$};
\draw[thick, color=blue](-1.5,-1) -- (-1.5, 1); 
     \node [below, color=blue] at (-1.5, -1) {$Z(x)$};
       \node[draw,circle, inner sep=1.5pt,color=black, fill=black] at (-1.5,0){};
       \node[draw,circle, inner sep=1.5pt,color=black, fill=black] at (1.5,0){};
      \node [above, color=black] at (-1.7,2) {$\boxed{\tilde{S}}$};
\end{scope}


      \draw[<-](2.5,0)--(3.5,0);
           \node [above, color=black] at (3,0) {$\pi_0$};
     
      \draw[<-](8.5,0)--(9.5,0);
             \node [above, color=black] at (9,0) {$\pi_1$};
     
      \draw[<-](14.5,0)--(15.5,0);
              \node [above, color=black] at (15,0) {$\pi_2$};
      
      \draw[<-](13.5,-4)--(16,-2);
            \node [below, color=black] at (14.9,-3.2) {$\gamma$};
      
       \draw[<-](2.5,-2)--(5.5,-4);
              \node [below, color=black] at (4,-3.2) {$\tilde{\pi}$};

\end{tikzpicture}
\end{center}
 \caption{The minimal embedded resolution $\pi_0 \circ \pi_1 \circ \pi_2$ 
     and a toric partial resolution $\tilde{\pi}$ of the cuspidal cubic $Y$ 
    contained in the plane $S_0 = \mathbb{C}^2$. The exceptional curve created by blowing up 
    the point $p_i$ is denoted by $E_i$ and its strict transform on the surface $S_j$ 
    is denoted by $E_{i,j}$. The strict transforms of $Z(x)$ and $Z(y)$ are still denoted by $Z(x)$ and $Z(y)$. 
    The strict transform of $Y$ on the surface $S_j$ is denoted by $Y_j$ and that on $\tilde{S}$ is 
    denoted by $\tilde{Y}$.}
\label{fig:embrescusp}
   \end{figure}
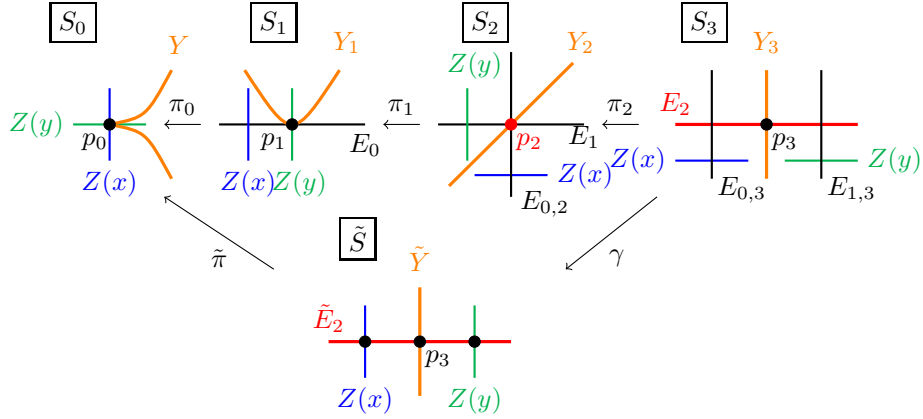

   The composed morphism $\pi_0 \circ \pi_1 \circ \pi_2 : S_3 \to S_0$ may be factored as 
       \[ \pi_0 \circ \pi_1 \circ \pi_2 = \tilde{\pi} \circ \gamma \]
   where 
      \[  \boxed{\tilde{\pi}} : \tilde{S} \to S_0 \]
   contracts the image $\boxed{\tilde{E}_2} := \gamma(E_2) \hookrightarrow \tilde{S}$ to the initial 
   point $p_0$ (see the left diagonal arrow of Figure \ref{fig:embrescusp}). Note that the restriction of $\gamma$ to the curve $E_2$ induces an 
   isomorphism from $E_2$ to $\tilde{E}_2$. For this reason, we still denote by $p_3$ the 
   image $\gamma(p_3)$.

 In  Figure \ref{fig:embrescusp} we have also represented the axes of coordinates 
  $Z(x), Z(y) \hookrightarrow \mathbb{C}^2= S_0$ and their strict transforms on the various surfaces 
  appearing in the drawing. For simplicity, we denote those strict transforms in the same way. 
  Note that the morphisms $\pi_i$ blow up singular points of the 
  total transforms of the divisor $Z(xy)$. This means that those morphisms are {\em toric}, 
  if one endows the affine plane $\mathbb{C}^2$ with its standard toric structure with dense torus 
  $(\mathbb{C}^*)^2$. As the morphisms $\pi_i$ are also birational, it means that they may be described by 
  successive  subdivisions of {\bf fans}, \index{Fan} that is, finite sets of strictly convex rational 
  cones inside the {\bf weight plane} $\mathbb{R}^2$ of the toric surface $\mathbb{C}^2$, closed under 
  the operation of taking faces and such that the intersection of any two of its cones is a common face. 
  The cones of such a fan may be of dimension {\em two}, of dimension {\em one} -- 
  in which case they are {\bf rays}  spanned by well-defined {\bf primitive vectors} of $\mathbb{Z}^2$, 
  that is, vectors of coprime coordinates -- or of dimension {\em zero} -- that is, equal to 
  the origin of the weight plane. A two-dimensional cone has three proper faces -- 
  two rays and the origin -- while a one-dimensional cone has only one proper face, namely, the origin. 
  
  Let us recall in which way a fan $\mathcal{F}$ contained in the weight plane $\mathbb{R}^2$ encodes a 
  {\bf toric surface} \index{Toric!surface} $\boxed{\mathcal{X}_{\mathcal{F}}}$ (see \cite{O 88}, \cite{F 93}, \cite{CLS 11} 
  or \cite[Section 1.3]{GGP 20}):
    \begin{enumerate}
       \item  \label{inclpresbij}
          There is an inclusion-preserving bijection between the cones $\tau$ of the fan 
           and the affine open toric subsurfaces $\boxed{\mathcal{X}_{\tau}}$ of $\mathcal{X}_{\mathcal{F}}$.  
           The set of $\mathbb{C}$-valued points of $X_\tau$ is the set of all \emph{monoid}
           homomorphisms from $\tau^\vee\cap\mathbb{Z}^2$ to $\mathbb{C}$, where $\tau^\vee$ is the
           dual cone of $\tau$ (see Section~\ref{sect:notations}) and $\mathbb{C}$ is viewed here as a
            \emph{multiplicative} monoid. Each such surface $\mathcal{X}_{\tau}$ contains 
           a unique closed orbit $\boxed{O_{\tau}}$, which is canonically a complex algebraic torus of 
           dimension $2 - \dim \tau$. It is also the unique smallest dimensional orbit of $\mathcal{X}_{\tau}$. 
           As a set, $O_\tau$ consists of those monoid homomorphisms which send
           $(\tau^\vee\smallsetminus \tau^\perp)\cap\mathbb{Z}^2$ to $0\in \mathbb{C}$, where
           $\tau^\perp$ is the orthogonal complement (or the annihilator) of $\tau$ in the dual vector space. 
       \item   \label{trivcone}
           To the trivial cone $0$ reduced to the origin of the weight plane $\mathbb{R}^2$ 
          corresponds an algebraic torus $\mathcal{X}_0$ isomorphic to $(\mathbb{C}^*)^2$. 
          It is Zariski-dense in 
          $\mathcal{X}_{\mathcal{F}}$. 
       \item   \label{raycone}
          To a one-dimensional cone, that is, a ray $\mathbb{R}_{\geq 0} w$ 
          spanned by a non-zero weight vector $w$ with 
           integral coordinates, corresponds a toric affine open surface 
           $\mathcal{X}_{\mathbb{R}_{\geq 0} w}$ 
           isomorphic to $\mathbb{C} \times \mathbb{C}^*$. 
        \item  \label{reg2dcone}
            To a {\bf regular} two-dimensional cone, that is, to a cone spanned by a basis of 
           the weight lattice  $\mathbb{Z}^2$, corresponds an affine open surface isomorphic 
           to $\mathbb{C}^2$. 
        \item   \label{nonreg2dcone}
           To a two-dimensional cone which is not regular corresponds an affine open surface 
           which is normal and has one singular point, namely, its $0$-dimensional orbit. 
        \item   \label{weightvect1paramgp}
            If $(p,q) \in \mathbb{Z}^2$ is an integer weight vector, then its corresponding 
               {\bf one parameter subgroup}  
              \[  \begin{array}{ccc}
                      \mathbb{C}^* & \to & (\mathbb{C}^*)^2 \\
                       t & \mapsto & (t^p, t^q) 
                  \end{array} \]
             has a limit inside $\mathcal{X}_{\mathcal{F}}$ when $t \to 0$ if and only if $(p,q)$ 
             belongs to the {\bf support} \index{Support!of a fan} 
             $\boxed{| \mathcal{F} |}$ of $\mathcal{F}$, that is, to the union of its cones. 
             This limit is the identity element  
             of the orbit $O_{\tau}$ seen as a group, where $\tau$ is the unique cone 
             of $\mathcal{F}$ whose relative interior contains $(p,q)$. 
    \end{enumerate}
  
  Note that if $\tau$ is a strictly convex rational cone of the weight plane $\mathbb{R}^2$ and $\mathcal{F}$ is the fan consisting of its faces, then one has $\mathcal{X}_{\mathcal{F}} = \mathcal{X}_{\tau}$.

    \begin{figure}[h!]
    \begin{center}
\begin{tikzpicture}[scale=0.5]

\begin{scope}[shift={(1,0)}]
        \fill[fill=yellow] (0,0) --(3.5,0)-- (3.5,3.5) -- (0, 3.5) --cycle;
        \draw [-, color=black, thick] (0,0)--(3.5,0);
             \node [below, color=blue] at (3.5,0) {$Z(x)$};
        \draw [-, color=black, thick] (0,0)--(0,3.5);
             \node [left, color=green!70!blue] at (0,3.5) {$Z(y)$};
  \node[draw,circle, inner sep=1.5pt,color=blue, fill=blue] at (1,0){};
  \node[draw,circle, inner sep=1.5pt,color=green!70!blue, fill=green!70!blue] at (0,1){};
           \node [left, color=black] at (2.5,1.5) {$p_0$};
\end{scope}


\begin{scope}[shift={(7,0)}]
       \fill[fill=yellow!40!white] (0,0) --(3.5,0)-- (3.5,3.5) -- (0, 3.5) --cycle;
        \fill[fill=yellow] (0,0) -- (3.5,3.5) -- (0, 3.5) --cycle;
        \draw [-, color=black, thick] (0,0)--(3.5,0);
             \node [below, color=blue] at (3.5,0) {$Z(x)$};
        \draw [-, color=black, thick] (0,0)--(0,3.5);
              \node [left, color=green!70!blue] at (0,3.5) {$Z(y)$};
                    \draw [-, color=black, thick] (0,0)--(3.5,3.5);
                         \node [right, above, color=black] at (3.5,3.5) {$E_0$};
  \node[draw,circle, inner sep=1.5pt,color=blue, fill=blue] at (1,0){};
  \node[draw,circle, inner sep=1.5pt,color=green!70!blue, fill=green!70!blue] at (0,1){};
   \node[draw,circle, inner sep=1.5pt,color=black, fill=black] at (1,1){};
            \node [left, color=black] at (1.8,2.5) {$p_1$};
\end{scope}


\begin{scope}[shift={(13,0)}]
         \fill[fill=yellow!40!white] (0,0) --(3.5,0)-- (3.5,3.5) -- (0, 3.5) --cycle;
           \fill[fill=yellow] (0,0) -- (3.5,3.5) -- (1.75, 3.5) --cycle;
        \draw [-, color=black, thick] (0,0)--(3.5,0);
               \node [below, color=blue] at (3.5,0) {$Z(x)$};
        \draw [-, color=black, thick] (0,0)--(0,3.5);
                \node [left, color=green!70!blue] at (0,3.5) {$Z(y)$};
          \draw [-, color=black, thick] (0,0)--(3.5,3.5);
                  \node [right, above, color=black] at (3.5,3.5) {$E_{0,2}$};
           \draw [-, color=black, thick] (0,0)--(1.75,3.5);
                   \node [left, above, color=black] at (1.75,3.5) {$E_1$};
  \node[draw,circle, inner sep=1.5pt,color=blue, fill=blue] at (1,0){};
  \node[draw,circle, inner sep=1.5pt,color=green!70!blue, fill=green!70!blue] at (0,1){};
     \node[draw,circle, inner sep=1.5pt,color=black, fill=black] at (1,1){};
      \node[draw,circle, inner sep=1.5pt,color=black, fill=black] at (1,2){};
                \node [left, color=black] at (2.7,2.9) {$p_2$};
\end{scope}


\begin{scope}[shift={(19,0)}]
         \fill[fill=yellow!40!white] (0,0) --(3.5,0)-- (3.5,3.5) -- (0, 3.5) --cycle;
        \draw [-, color=black, thick] (0,0)--(3.5,0);
               \node [below, color=blue] at (3.5,0) {$Z(x)$};
        \draw [-, color=black, thick] (0,0)--(0,3.5);
                \node [left, color=green!70!blue] at (0,3.5) {$Z(y)$};
                  \draw [-, color=black, thick] (0,0)--(3.5,3.5);
                         \node [right, above, color=black] at (3.5,3.5) {$E_{0,3}$};
                   \draw [-, color=black, thick] (0,0)--(1.75,3.5);
                           \node [left, above, color=black] at (1.5,3.5) {$E_{1,3}$};
                    \draw [-, color=red, thick] (0,0)--(2.33,3.5);
                             \node [left, above, color=red] at (2.45,3.5) {$E_2$};
  \node[draw,circle, inner sep=1.5pt,color=blue, fill=blue] at (1,0){};
  \node[draw,circle, inner sep=1.5pt,color=green!70!blue, fill=green!70!blue] at (0,1){};
     \node[draw,circle, inner sep=1.5pt,color=black, fill=black] at (1,1){};
      \node[draw,circle, inner sep=1.5pt,color=black, fill=black] at (1,2){};
            \node[draw,circle, inner sep=1.5pt,color=red, fill=red] at (2,3){};
\end{scope}


\begin{scope}[shift={(9.5,-6)}]
          \fill[fill=yellow!40!white] (0,0) --(3.5,0)-- (3.5,3.5) -- (0, 3.5) --cycle;
        \draw [-, color=black, thick] (0,0)--(3.5,0);
               \node [below, color=blue] at (3.5,0) {$Z(x)$};
        \draw [-, color=black, thick] (0,0)--(0,3.5);
               \node [left, color=green!70!blue] at (0,3.5) {$Z(y)$};
                     \draw [-, color=red, thick] (0,0)--(2.33,3.5);
                           \node [left, above, color=red] at (2.45,3.5) {$\tilde{E}_2$};
  \node[draw,circle, inner sep=1.5pt,color=blue, fill=blue] at (1,0){};
  \node[draw,circle, inner sep=1.5pt,color=green!70!blue, fill=green!70!blue] at (0,1){};
       \node[draw,circle, inner sep=1.5pt,color=red, fill=red] at (2,3){};
\end{scope}


      \draw[<-](5,2)--(6,2);
     
      \draw[<-](11,2)--(12,2);
     
      \draw[<-](17,2)--(18,2);
      
      \draw[<-](15,-3)--(18,-1);
      
       \draw[<-](4,-1)--(7,-3);

\end{tikzpicture}
\end{center}
 \caption{The fans corresponding to the toric surfaces of Figure \ref{fig:embrescusp} and their subdivisions
corresponding to the toric morphisms of the same figure. We write the name of the
corresponding irreducible component of the total transform of $Z(xy) \hookrightarrow S_0$. 
Inside each regular  two-dimensional cone which gets subdivided, 
we indicate the name of the corresponding $0$-dimensional orbit.  The right diagonal arrow is the toric modification which blows down the orbit closures $E_{0,3}$ and $E_{1,3}$. The left diagonal arrow blows then down the orbit closure $\tilde{E}_2$.}
\label{fig:fansforcusp}
   \end{figure}
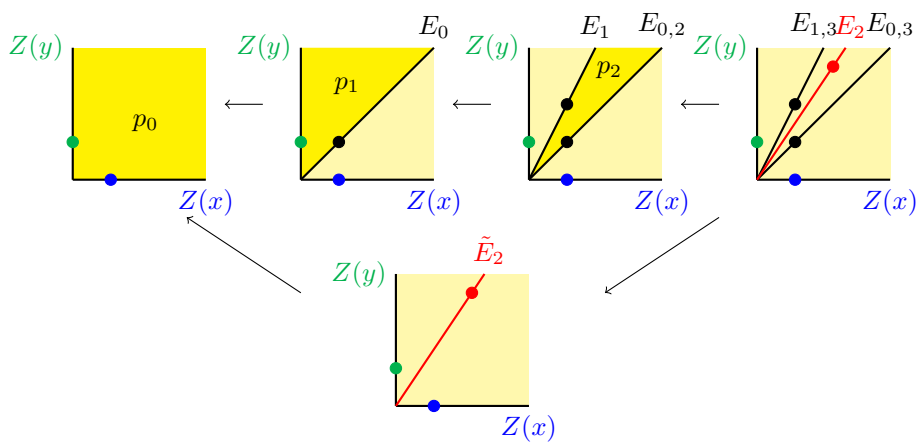

The process of successive subdivisions of fans corresponding to the toric morphisms 
of  Figure \ref{fig:embrescusp} is represented in Figure \ref{fig:fansforcusp}. 
The starting fan $\boxed{\mathcal{F}_0}$, corresponding to the 
surface $S_0  = \mathbb{C}^2$, consists of the non-negative quadrant $\boxed{\sigma_0}$ 
of the weight plane $\mathbb{R}^2$ 
of $\mathbb{C}^2$ and of its faces. Near each ray $\tau$ we write the name of the irreducible component 
of the total transform of $Z(xy)$ which is the closure of the corresponding one-dimensional orbit $O_{\tau}$.
That irreducible component is isomorphic either to $\mathbb{C}$ (in the case of the strict transforms of 
$Z(x)$ and $Z(y)$) or to $\mathbb{C}\mathbb{P}^1$ (in the case of an exceptional divisor $E_{i,j}$). On the
upper horizontal line of  Figure \ref{fig:fansforcusp}, each blowup is represented combinatorially by the
subdivision of a regular cone $\tau$ along a ray. Inside each such cone $\tau$ we indicate the name of
the associated point, which is the $0$-dimensional orbit of $\mathcal{X}_{\tau}$. The subdividing ray is
spanned by the sum $w_i + w_j$ of the primitive weight vectors $w_i$ and $w_j$ spanning the edges of
$\tau$. All those primitive vectors are drawn as small dots. 
  
  Let us consider inside each surface $S_1, S_2, S_3$ appearing in the upper part of Figure \ref{fig:embrescusp} the affine toric open set $\mathcal{X}_{\tau}$ isomorphic to $\mathbb{C} \times \mathbb{C}^*$ which corresponds to the ray $\tau$ which was added to get the surface. The strict transforms of the curve $Y$ inside those affine toric surfaces $\mathcal{X}_{\tau}$ are the intersections of
$\mathcal{X}_{\tau}$ with the curves $Y_1, Y_2, Y_3$, that is, $Y_1 \smallsetminus p_1, Y_2 \smallsetminus p_2$ and  $Y_3$ respectively. Indeed, only $Y_3$ does not pass through a singular point of the total transform of the divisor $Z(xy)$, that is, through a $0$-dimensional orbit of the ambient toric surface, and $\mathbb{C} \times \mathbb{C}^*$ does not contain such an orbit. Restricting the toric birational morphisms $S_i \to S_0$ to those three strict transforms, we get  the restrictions: 
  \[   \pi_1 | \colon Y_1 \smallsetminus p_1 \to Y, \  \   
      \pi_1 \circ \pi_2 | \colon Y_2 \smallsetminus p_2 \to Y, \  \ 
      \pi_1 \circ \pi_2 \circ \pi_3 | \colon Y_3 \to Y.  \] 
   Only the last morphism is {\em proper}, similarly to what happens for the morphism 
$\tilde{\pi} : \tilde{S} \to S_0$ (see the bottom of Figure \ref{fig:embrescusp}). The curve $Y_3$ is the strict transform of $Y$ in the affine toric surface $\mathcal{X}_{\mathbb{R}_{\geq 0} (2,3)}$ determined by the ray $\mathbb{R}_{\geq 0} (2,3)$ spanned by the weight vector $(2,3)$. If we want to also turn the first two morphisms into proper ones, we need to add the orbits $p_1$ and $p_2$ respectively, that is, to look inside the affine toric surfaces $\mathcal{X}_{\tau}$, where $\tau$ is one of the two-dimensional regular cones decorated by $p_1$ and $p_2$ in Figure \ref{fig:fansforcusp}. Note that both cones contain the ray $\mathcal{X}_{\mathbb{R}_{\geq 0} (2,3)}$. 
  
  In fact, {\em for every fan $\mathcal{F}$ whose support is contained in $\sigma_0$, the restriction 
    $\pi \colon Y_{\pi} \to Y$ of the toric morphism $\pi_{\mathcal{F}} \colon \mathcal{X}_{\mathcal{F}} \to \mathcal{X}_{\sigma_0} = \mathbb{C}^2$ to the strict transform $Y_{\pi}$ in $\mathcal{X}_{\mathcal{F}}$ is proper if and only if $| \mathcal{F} |$ contains the ray $\mathbb{R}_{\geq 0} (2,3)$.} Therefore, according to the toric-geometrical interpretation of local tropicalization explained in Section \ref{sect:intro}, {\em this ray is the local tropicalization of the germ $(Y, p_0)$.}

\medskip
{\bf Question:} {\em How is the ray $\mathbb{R}_{\geq 0} (2,3)$ related to the defining series $f(x,y)$ 
     of the germ $(Y, p_0)$ (see formula \eqref{eq:defcubic})?}
\medskip

{\bf Answer:}  It is the ray orthogonal to the unique compact edge of the Newton polygon of $f$. 
\medskip

Proposition \ref{prop:torcharltroptwod} below states the analogous result for all plane curve singularities without branches contained in the coordinate axes.

\medskip
\section{The Newton polygon of a plane curve singularity}  \label{sect:Newtonpolyg}
\medskip

In this section we relate the two viewpoints on local tropicalization explained in the introduction -- 
via weight vectors of parametrizations and toric geometry -- 
in the case of plane curve singularities $Y \hookrightarrow (\mathbb{C}^2,0)$ defined by series 
$f(x,y) \in \mathbb{C}[[x,y]]$ which are divisible neither by $x$ nor by $y$. 
The object allowing to make this connection precise is the {\em Newton polygon} $\mathcal{N}(f)$ 
of $f$ (see Definition \ref{def:NP}, Corollary \ref{cor:arcweightplanecurve} and Proposition \ref{prop:torcharltroptwod}). 
We end the section by introducing a third viewpoint on the local tropicalization, 
using the {\em initial forms} $\mathrm{in}_{(p,q)} f$ of $f$ (see Corollary \ref{cor:reforminit}).

\medskip

Consider a {\bf plane curve singularity} \index{Singularity!plane curve} $Y \hookrightarrow (\mathbb{C}^2, 0)$, which for us means a reduced formal germ  
of dimension one. Since the local ring  $\mathbb{C}[[x,y]]$ is factorial, 
we know that $Y$ is a principal divisor, 
that is, there exists $f(x,y) \in \mathbb{C}[[x,y]] \smallsetminus \{0\}$ satisfying $f(0, 0) = 0$, 
such that $Y = Z(f)$.  The series $f(x,y)$ is well-defined 
up to multiplication by a unit of $\mathbb{C}[[x,y]]$, that is, by a series with non-zero constant term. 
The germ $Y$ consists of a finite number of {\bf branches}, \index{Branch} 
that is, irreducible germs of curves. 

We assume in the sequel that {\em no branch of} $Y$ {\em is contained in the union of 
coordinate axes $Z(x)$ and $Z(y)$ of} $\mathbb{C}^2$, which amounts to say that $f$ is not divisible 
by $x$ or $y$. Note that in terms of Definition \ref{def:subtorgerm} below, this means that $Y$ 
is an {\em internal subtoric germ} of $(\mathbb{C}^2,0)$.

A basic property of plane {\em branches} $B$ is that their normalizations are {\em smooth}, 
which means that they are described by parametrizations $t \mapsto (x(t), y(t))$, with $x(t), y(t) \in \mathbb{C}[[t]]$. Special kinds of such normalizing parametrizations may be obtained by the classical iterative
method of Newton (see \cite[pages 59-60]{G 17} or \cite[Section 1.2.6]{GGP 20}). They satisfy the property
that $x(t) = t^{n}$, where $n \in \mathbb{Z}_{>0}$ is the intersection number at $0$ of $B$ with the branch
$(Z(x),0)$. Each step of Newton's method uses a suitable {\em Newton polygon}. 

Let us explain how Newton polygons pop up naturally in this context. Instead of looking for 
{\em Newton-Puiseux parametrizations} of the form $t \mapsto (t^n, y(t))$, 
we will consider an arbitrary parametrization $t \mapsto (x(t), y(t))$ of a branch of $Y = Z(f)$ 
and we will look for constraints on its {\em initial terms}.  As no branch of $Y$ is contained 
in the coordinate axes of $\mathbb{C}^2$, both series $x(t)$ and $y(t)$ are non-zero. We have the relation 
\begin{equation}\label{eq:idzero}
        f(x(t),y(t)) = 0
\end{equation}
in the ring $\mathbb{C}[[t]]$. Let us write
     \begin{equation}\label{eq:expofxy}
          x(t)=\sum_{i= p}^{\infty} \boxed{x_i} t^i, \quad y(t)=\sum_{j= q}^{\infty} \boxed{y_j} t^j, 
               \quad x_i,y_j\in \mathbb{C},
      \end{equation}
where $p>0$, $q>0$, $x_p\ne 0$, $y_q\ne 0$. Therefore $x_p t^p$ and $y_q t^q$ are the {\bf initial terms} \index{Initial!term} of $x(t)$ and $y(t)$.

\begin{definition}  \label{def:supptwod}
    The {\bf support} \index{Support!of a series} $\boxed{\mathrm{supp}\ f}$ of $f \in 
    \mathbb{C}[[x,y]] \smallsetminus \{0\}$ 
     is the set of exponents $(\alpha,\beta)\in \mathbb{Z}_{\geq 0}^{2}$ such that the
    coefficient $f_{\alpha,\beta}$ of the monomial $x^\alpha y^\beta$ in the expansion of $f$ does not vanish. With this notation we can write
     \begin{equation}\label{eq:expoff}
               f(x,y)=\sum_{(\alpha,\beta) \in \  \mathrm{supp} \ f} \boxed{f_{\alpha,\beta}}  \, 
                x^\alpha y^\beta, \quad f_{\alpha,\beta}  \in \mathbb{C}^* 
                \ \mbox{ for all } \ (\alpha,\beta) \in  \mathrm{supp}\  f.
      \end{equation}
\end{definition}
Substituting  \eqref{eq:expofxy} into \eqref{eq:expoff} and rearranging the terms, 
the identity \eqref{eq:idzero} becomes 
       \begin{equation}   \label{eq:rearr}
             \left( \sum_{(\alpha,\beta)\in \  \delta(p,q)} f_{\alpha, \beta} \  x_{p}^{\alpha} y_{q}^{\beta} \right) 
                   \cdot t^{m(p,q)} +   \sum_{m> m(p,q)} d_m t^m = 0,
        \end{equation}
for some coefficients $d_m\in \mathbb{C}$.   The set $\delta(p,q) \subseteq \mathrm{supp}\  f$ and the positive integer $m(p,q)$ appearing in \eqref{eq:rearr} are defined as follows, recalling 
that $\sigma_0$ denotes the non-negative quadrant of $\mathbb{R}^2$: 

\begin{definition}    \label{def:minlinform}
      Consider $(p,q) \in \sigma_0^{\circ} \cap \mathbb{Z}^2 = \mathbb{Z}_{>0}^2$ 
      and $f \in \mathbb{C}[[x,y]] \smallsetminus \{0\}$.
     The {\bf basis  $\boxed{\delta_{(p,q)}(f)}$ relative to $(p,q)$} of the support
     $\mathrm{supp}\  f$ of $f$  \index{Basis!of a support}
     is the subset of $\mathrm{supp}\  f$ on which the linear function 
         \[  \begin{array}{cccc}
                 \boxed{(p,q)} : &   \mathbb{R}^2 &   \to  & \mathbb{R}, \   \\
                   &   (\alpha,\beta) &   \mapsto &   p \alpha+ q \beta 
              \end{array}
         \]
      achieves its minimum $\boxed{m(p,q)} \in \mathbb{Z}_{\geq 0}$ when restricted to $\mathrm{supp}\  f$. 
 \end{definition}

 As $(\alpha, \beta)$ is the {\em vector of exponents} of the monomial $x^{\alpha} y^{\beta}$, 
 the pair $(p,q)$ may be interpreted as a {\em weight vector} relative to the variables $(x,y)$. We chose the name of {\em basis of $\mathrm{supp}\  f$} for $\delta_{(p,q)}(f)$ as we think about it as the basis on which $\mathrm{supp}\  f$ lies when it is studied using the {\em height} function $(p,q) : \mathbb{R}^2 \to \mathbb{R}$.
 
 The following classical result explains how convex geometry introduces itself at this precise moment:
 
 \begin{proposition}   \label{prop:congeomapp}
      Let $F$ be a finite subset of the real vector space $V$ and let $\ell : V \to \mathbb{R}$ be a 
      linear form. Then, the minimum of the restriction of $\ell$ to $F$ is achieved 
      precisely on the intersection of $F$ with a face of its own convex hull. Moreover, all the faces 
      of the convex hull of $F$ appear when $\ell$ varies in the dual vector space of $V$. 
 \end{proposition}
 
For a summary of facts from convex geometry used in the theory of toric varieties see Oda's 
\cite[Appendix]{O 88}.
 
In our context, we have to apply Proposition \ref{prop:congeomapp} to $F := \mathrm{supp}\  f$.
As $f$ is a series, its support $\mathrm{supp}\  f$ may be {\em infinite}. Nevertheless, 
it is easy to see that if $p >0, q>0$, that is, if $(p,q)$ varies in the interior $\sigma_0^{\circ}$ 
of the non-negative quadrant $\sigma_0$ of the plane $\mathbb{R}^2$ of weight vectors, 
then the restriction of the linear form $(p,q)$ of Definition \ref{def:minlinform} to 
 $\mathrm{supp}\  f$ achieves its minimum on a {\em finite} set. Indeed, the possible such sets 
 are the intersections of $\mathrm{supp}\  f$ with 
 the {\em compact} faces of a variant of the convex hull of $\mathrm{supp}\  f$, which is called 
 the {\em Newton polygon} of $f$. This polygon is defined as follows:

\begin{definition}\label{def:NP}
   The  {\bf Newton polygon $\boxed{\mathcal{N}(f)}$  of the series}  \index{Newton!polygon}
   $f \in \mathbb{C} [[x,y]]  \smallsetminus \{0\}$ is 
     the convex hull in $\mathbb{R}^2$ of the set
       \[ \mathrm{supp}\  f + \mathbb{R}_{\geq 0}^{2}  \ = \ \bigcup_{(\alpha,\beta)\in \ \mathrm{supp}\  f} \big( (\alpha,\beta)+\mathbb{R}_{\geq 0}^{2} \big). \] 
\end{definition}

 The sum of the left-hand side of the previous equality is the so-called {\bf Minkowski sum} 
 \index{Minkowski sum} of subsets of a real vector space, defined by:
    \[ A + B := \{ a + b, \  a \in A, b \in B \}.   \]

Our assumption that $f$ is not divisible by $x$ or $y$ means that its Newton polygon $\mathcal{N}(f)$ 
intersects both axes of the plane of exponents of monomials. 
Its compact faces are either {\em vertices} or {\em edges}.

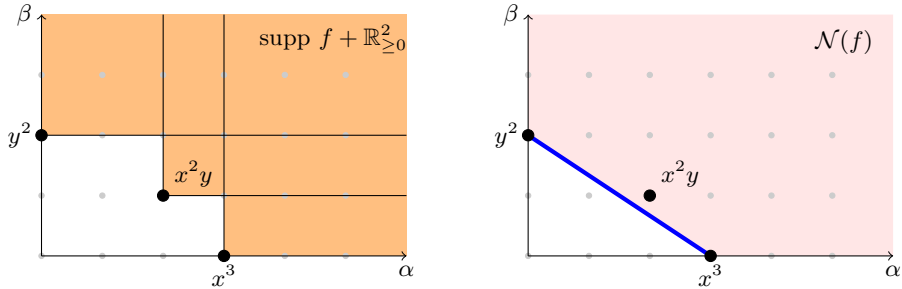
\begin{figure}[h!] 
\begin{center}
\begin{tikzpicture}[scale=0.8]

\begin{scope}[shift={(0,0)}]

\tikzstyle{every node}=[font=\small]

    \fill[fill=orange!50!white] (0,4) --(6,4)-- (6,2) -- (0,2) --cycle;
     \fill[fill=orange!50!white] (2,1) --(2,4)-- (6,4) -- (6,1) --cycle;
      \fill[fill=orange!50!white] (3,0) --(3,4)-- (6,4) -- (6,0) --cycle;
\foreach \x in {0,1,...,5}{
\foreach \y in {0,1,...,3}{
       \node[draw,circle,inner sep=0.7pt,fill, color=gray!40] at (1*\x,1*\y) {}; }
   }
\draw[->] (0,0) -- (6,0) node[right,below] {$\alpha$};
\draw[->] (0,0) -- (0,4) node[above,left] {$\beta$};

\draw[-] (3,4) -- (3,0)  ;
\draw[-] (2,4) -- (2,1) -- (6,1) ;
\draw[-] (0,2) -- (6,2)  ;

\node[draw,circle,inner sep=1.5pt,fill=black, color=black] at (3,0) {};
\node[draw,circle,inner sep=1.5pt,fill=black,color=black] at (0,2) {};
\node[draw,circle,inner sep=1.5pt,fill] at (2,1) {};
\node [below] at (3,0) {$x^{3}$};
\node [left] at (0,2) {$y^{2}$};
\node [above] at (2.5, 1) {$x^2 y$};
\node [above] at (4.8, 3.2) {$\mathrm{supp}\  f +\mathbb{R}_{\geq 0}^{2}$};

\end{scope}


\begin{scope}[shift={(8,0)}]

\tikzstyle{every node}=[font=\small]
\fill[fill=pink!40!white] (0,4) --(6,4)-- (6,0) -- (3, 0) -- (0,2) --cycle;
 
\foreach \x in {0,1,...,5}{
\foreach \y in {0,1,...,3}{
       \node[draw,circle,inner sep=0.7pt,fill, color=gray!40] at (1*\x,1*\y) {}; }
   }
\draw[->] (0,0) -- (6,0) node[right,below] {$\alpha$};
\draw[->] (0,0) -- (0,4) node[above,left] {$\beta$};
\draw [ultra thick, color=blue](3,0) -- (0,2);
\node[draw,circle,inner sep=1.5pt,fill=black, color=black] at (3,0) {};
\node[draw,circle,inner sep=1.5pt,fill=black,color=black] at (0,2) {};
\node[draw,circle,inner sep=1.5pt,fill] at (2,1) {};
\node [below] at (3,0) {$x^{3}$};
\node [left] at (0,2) {$y^{2}$};
\node [above] at (2.5, 1) {$x^2 y$};
\node [above] at (5.2, 3.2) {$\mathcal{N}(f)$};

\end{scope}

 \end{tikzpicture}
\end{center}
\caption{The Minkowski sum $\mathrm{supp}\  f +\mathbb{R}_{\geq 0}^{2}$ of Definition \ref{def:NP}
     (on the left)
    and the Newton polygon $\mathcal{N}(f)$ of the series $f(x,y)=y^2 - 2 x^3 + x^2 y$ 
    (on the right).}  
   \label{fig:NPcubicex}
    \end{figure}

\begin{example}  \label{ex:NPcubic}
     Let us consider again the series $f(x,y) := y^2 - 2 x^3 + x^2 y$ of formula  \eqref{eq:defcubic}. 
     Figure \ref{fig:NPcubicex} shows the Minkowski sum 
     $\mathrm{supp}\  f +\mathbb{R}_{\geq 0}^{2}$  of Definition
     \ref{def:NP} and its convex hull, which is by definition the Newton polygon $\mathcal{N}(f)$. 
  \end{example}

     \begin{remark} 
     Assume that $f \in \mathbb{C}[x,y]$, as in Example \ref{ex:NPcubic}.
     If instead of weight vectors with positive coordinates we had considered arbitrary 
     non-zero vectors $(p,q)$ of the weight plane then, seen as linear forms on the plane 
     of exponents of monomials, they would have achieved their minima on the 
     intersections of $\mathrm{supp}\  f$ with the faces of the {\em convex hull of this support}. 
     This convex hull -- a triangle in Example \ref{ex:NPcubic} -- is by definition the 
     {\em global Newton polygon} of the polynomial $f(x,y)$, which gives information 
     about the parametric curves contained in $Z(f)$ which are described by 
     {\em Laurent power series} (with possibly negative exponents) instead of 
     power series with non-negative exponents. In this paper we are concerned only 
     with parametric curves {\em passing through the origin of $\mathbb{C}^2$ or, 
     more generally, of $\mathbb{C}^n$}. 
     That is why we define the Newton polygon of $f$ as in Definition \ref{def:NP}. In order 
     to differentiate it from the convex hull of the support, one may call it the 
     {\em local Newton polygon}. Note that the term \emph{Newton diagram} \index{Newton!diagram}
     is also met in the literature, sometimes meaning the strict subset of the local 
     Newton polygon which is the union of its compact faces (see for instance \cite{V 76}). 
     \end{remark}

The map associating to each non-zero series its Newton polygon behaves like the 
logarithm function, in the sense that it transforms products into Minkowski sums (see \cite[Proposition 1.6.13]{GGP 20}):

\begin{proposition} \label{prop:NPMinksum}
     Assume that $f,g \in \mathbb{C}[[x,y]] \smallsetminus \{0\}$. Then:
         \[   \mathcal{N}(fg) = \mathcal{N}(f) + \mathcal{N}(g). \]
\end{proposition}

    Proposition \ref{prop:NPMinksum} implies that $ \mathcal{N}(fg) = \mathcal{N}(f)$ whenever $g$ is a unit of $\mathbb{C}[[x,y]]$. Therefore the Newton polygon does not depend of the defining series of a given plane curve singularity $Y \hookrightarrow \mathbb{C}^2$. This allows to simply speak of the {\bf Newton polygon  \index{Newton!polygon} $\boxed{\mathcal{N}(Y)}$ of $Y$}. Even if we drop it for simplicity from the notation, it is not an invariant of the abstract germ $Y$, since it depends strongly on the chosen embedding of $Y$ in $(\mathbb{C}^2,0)$ (see \cite{O 83} or \cite{GLP 07}).

Let us come back to the equality \eqref{eq:rearr}. For this equality to be possible, it is necessary that
{\em the coefficient} 
   \[  \boxed{d_{m(p,q)}} :=  \sum_{(\alpha,\beta)\in \ \delta(p,q)} f_{\alpha, \beta} \  x_{p}^{\alpha} y_{q}^{\beta}  \] 
{\em of $t^{m(p,q)}$ contains at least two summands}. That is, {\em the face $\delta(p,q)$ must be a compact edge of $\mathcal{N}(f)$ and not a vertex}. If this is the case, then, since $\mathbb{C}$ is algebraically closed, it is indeed possible
to choose the initial coefficients in the series $x(t)$ and $y(t)$ so that $d_{m(p,q)} =0$. 
Moreover, as shown for instance in \cite[Section IV.3.2]{W 50} or
\cite[Section~2.1]{C 04}, one can organize in this way an iterative process which 
computes parametrizations of all the branches of $Z(f)$. 

To say that $(p,q) \in \sigma^{\circ}_0$ {\em achieves its minimum} on the edge $\delta(p,q)$ of 
$\mathcal{N}(f)$ means that the ray $\mathbb{R}_{\geq 0} (p,q)$ is {\em orthogonal} to  $\delta(p,q)$. 
We may thus summarize the previous explanations as follows: 

\begin{theorem}\label{thm:pqpair}
     For a given plane curve singularity $Y \hookrightarrow (\mathbb{C}^2,0)$ without branches contained 
     in the coordinate axes, there exists only a finite number of possibilities for
     the quotient $q/p$ of initial exponents $p$ and $q$ of a parametrized  
     arc contained in $Y$. Namely, those quotients are exactly the slopes 
     of the rays orthogonal to the compact edges of the Newton polygon $\mathcal{N}(Y)$ of $Y$.
\end{theorem}

\begin{remark}   \label{rem:notprim}
     It is not always possible to choose the pair $(p,q)$ of Theorem \ref{thm:pqpair} 
      to be primitive, that is, such that 
     $p$ and $q$ are coprime.  For instance, let $f(x,y) \in \mathbb{C}[[x,y]]$ be the minimal 
     polynomial in $\mathbb{C}[[x]][y]$ of the Newton-Puiseux series $x^{3/2} + x^{7/4}$. Therefore, 
     $(x(t) := t^4, \  y(t) := t^6 + t^7)$ is a parametrization 
     of the branch $Y \hookrightarrow \mathbb{C}^2$ defined by $f$. 
     The weight vector $(4, 6)$ of initial exponents of the series $(x(t), y(t))$
     is not primitive and $(2,3)$ is the unique primitive weight vector with positive 
     coordinates which is proportional to it. Assume that a pair 
     $(x_1(t), y_1(t)) \in \mathbb{C}[[t]]$ has $(2, 3)$ as weight vector of initial exponents. 
     Then, one may perform a change of 
     variable, replacing $t$ by $u \in \mathbb{C}[[t]]$ such that $x_1(t) = u^2$. 
     The series $y_1(t)$ becomes a 
     series $y_2(u) \in \mathbb{C}[[u]]$. The pair $(u^2, y_2(u))$ parametrizes the same branch 
     $B \hookrightarrow \mathbb{C}^2$ as $(x_1(t), y_1(t))$ therefore, as $2$ is a prime number,  
     this branch has at most one {\em characteristic exponent} (recall that the {\em characteristic 
     exponents of a Newton-Puiseux series} are those exponents of its support which cannot be 
     written as a fraction whose denominator is the lowest common denominator of the exponents 
     strictly lower than them, see \cite[Definition 1.6.2]{GGP 20}). 
     This shows that $B$ cannot be equal to the branch $Y$, which has 
     two characteristic exponents, namely $3/2$ and $7/4$. 
\end{remark}

Using the first definition of the {\em local tropicalization} of $Y \hookrightarrow (\mathbb{C}^2,0)$ given in the introduction, in terms of exponent vectors of arcs contained in  $Y$, Theorem \ref{thm:pqpair}  implies:

\begin{corollary}   \label{cor:arcweightplanecurve}
    Let $Y \hookrightarrow (\mathbb{C}^2,0)$ be a plane curve singularity without branches contained 
     in the coordinate axes. Then its local tropicalization is the union of the rays orthogonal 
     to the compact edges of the Newton polygon $\mathcal{N}(Y)$ of $Y$. 
\end{corollary}

 \begin{figure}[h!]
    \begin{center}
\begin{tikzpicture}[scale=1]

          \fill[fill=yellow!40!white] (0,0) --(3.5,0)-- (3.5,3.5) -- (0, 3.5) --cycle;
        \draw [-, color=black, thick] (0,0)--(3.5,0);
        \draw [-, color=black, thick] (0,0)--(0,3.5);
                     \draw [-, color=red, thick] (0,0)--(2.33,3.5);
       
   \foreach \x in {0,1,...,3}{
\foreach \y in {0,1,...,3}{
       \node[draw,circle,inner sep=0.7pt,fill, color=gray!40] at (1*\x,1*\y) {}; }
       
     \node[draw,circle, inner sep=1.5pt,color=red, fill=red] at (2,3){};
   }

\end{tikzpicture}
\end{center}
 \caption{The local tropicalization of the series $f(x,y)=y^2 - 2 x^3 + x^2 y$ 
    is the ray spanned by the primitive vector $(2,3)$, which is indicated with a filled dot.}
\label{fig:loctropcubic}
   \end{figure}
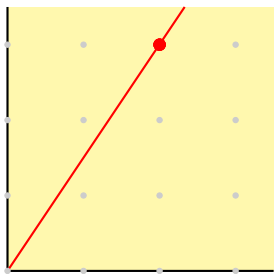

    The local tropicalization of the series of Example\ref{ex:NPcubic} 
    is represented in Figure \ref{fig:loctropcubic}. One recognizes the diagram of the bottom of 
   Figure \ref{fig:fansforcusp}. This is not a coincidence. Indeed, 
   one has the following property, which is illustrated in Figures \ref{fig:embrescusp}  
   and \ref{fig:fansforcusp}:
   
   \begin{proposition}   \label{prop:minfan}
        Let $Y \hookrightarrow (\mathbb{C}^2,0)$ be a plane curve singularity without branches contained 
        in the coordinate axes.  
        Let $\mathcal{F}$ be a fan with support contained in the non-negative quadrant $\sigma_0$ 
        of the weight plane $\mathbb{R}^2$. Consider the strict transform $Y_{\mathcal{F}}$ of $Y$ by the 
        toric morphism $\pi_{\mathcal{F}} : \mathcal{X}_{\mathcal{F}} \to \mathbb{C}^2 = 
        \mathcal{X}_{\sigma_0}$. Then,  $Y_{\mathcal{F}}$ 
        avoids the $0$-dimensional orbits of the toric surface $\mathcal{X}_{\mathcal{F}}$ if and only if 
        the rays orthogonal to the compact edges of the Newton polygon $\mathcal{N}(Y)$ of $Y$ 
        are cones of the fan $\mathcal{F}$. 
   \end{proposition}

\begin{proof}
    We add cones to the fan $\mathcal{F}$ until we obtain a fan $\mathcal{F}'$ 
    whose support is $\sigma_0$. Then, we
    subdivide the two-dimensional cones of $\mathcal{F}'$ using the rays orthogonal to the compact edges 
    of the Newton polygon $\mathcal{N}(Y)$ of $Y$, 
    obtaining a fan $\hat{\mathcal{F}}$. By \cite[Proposition 1.4.18]{GGP 20}, 
    the strict transform $Y_{\hat{\mathcal{F}}}$ 
    of $Y$ by the toric morphism $\pi_{\hat{\mathcal{F}}} : \mathcal{X}_{\hat{\mathcal{F}}} 
    \to \mathbb{C}^2 = \mathcal{X}_{\sigma_0}$ intersects 
    the toric boundary  $\partial \mathcal{X}_{\hat{\mathcal{F}}}$  only at points 
    of the orbits of dimension one 
    which correspond to the rays orthogonal to the compact edges 
    of $\mathcal{N}(Y)$. As a consequence, the strict transform $Y_{\mathcal{F}}$ intersects the toric 
    boundary $\partial \mathcal{X}_{\mathcal{F}}$ only at points of the images of those orbits by the toric 
    morphism $\hat{\pi} : \mathcal{X}_{\hat{\mathcal{F}}} \to \mathcal{X}_{\mathcal{F}}$. 
    Those images are precisely the 
    orbits of $\mathcal{X}_{\mathcal{F}}$ which correspond to the cones $\sigma$ of $\mathcal{F}$ 
    whose relative interiors contain 
    rays orthogonal to the compact edges of $\mathcal{N}(Y)$. To say that no such orbit is of dimension $0$ 
    means that no such cone $\sigma$ is of dimension two. This is equivalent to the fact that 
    the rays orthogonal to the compact edges of the Newton polygon of $Y$ 
    are cones of the fan $\mathcal{F}$. 
\end{proof}

The previous proof yields also the following result:

\begin{proposition}   \label{prop:torcharltroptwod}
        Let $Y \hookrightarrow (\mathbb{C}^2,0)$ be a plane curve singularity without branches contained 
        in the coordinate axes.  
        Let $\mathcal{F}$ be a fan with support contained in the non-negative quadrant $\sigma_0$ 
        of the weight plane $\mathbb{R}^2$. Consider the strict transform $Y_{\mathcal{F}}$ of $Y$ by the 
        toric morphism $\pi_{\mathcal{F}} : \mathcal{X}_{\mathcal{F}} \to 
        \mathbb{C}^2 = \mathcal{X}_{\sigma_0}$. Then  the 
        restriction $\pi : Y_{\mathcal{F}} \to Y$ of $\pi_{\mathcal{F}} $ to $Y_{\mathcal{F}}$ 
        is proper if and only if 
        the support of $\mathcal{F}$ contains  the rays orthogonal to the 
        compact edges of the Newton polygon $\mathcal{N}(Y)$ of $Y$.
\end{proposition}
   
   This proves the equivalence of the two characterizations of local tropicalization given in 
   Section \ref{sect:intro} in the case of plane curve singularitues $Y \hookrightarrow (\mathbb{C}^2,0)$ 
   without branches contained in the coordinate axes.

\medskip
We have worked above with initial terms of univariate series (see the equalities \eqref{eq:expofxy}). 
There is a corresponding notion for bivariate series, this time {\em relative to a weight vector}: 

\begin{definition}   \label{def:initform}
  Let $(p,q) \in \sigma_0^{\circ} \cap \mathbb{Z}^2 = \mathbb{Z}_{>0}^{2}$ and 
  $f \in \mathbb{C}[[x,y]] \smallsetminus \{0\}$.  The {\bf $(p,q)$-initial form} 
  of the series $f$ \index{Initial!form of a series} is 
  the $(p,q)$-weighted homogeneous polynomial
      \[ \boxed{\mathrm{in}_{(p,q)} f} :=\sum_{(\alpha,\beta)\in \delta_{(p,q)}(f)} 
      f_{\alpha,\beta}\, x^\alpha y^\beta  
          \ \ \in \mathbb{C}[x,y]. \]
   Here $\delta_{(p,q)}(f) \subseteq \mathrm{supp}\  f$ denotes the basis 
   of $\mathrm{supp}\  f$ relative to $(p,q)$,  introduced in Definition \ref{def:minlinform}. 
\end{definition}

This notion allows to rephrase Theorem \ref{thm:pqpair} using a defining series of the given plane curve
singularity, instead of parametrized arcs contained in it: 

\begin{corollary}   \label{cor:reforminit}
    Let $(p,q)\in  \mathbb{Z}_{>0}^{2}$ and $f \in \mathbb{C}[[x,y]] \smallsetminus \{0\}$. 
    Then, $(p,q)$ belongs to the local tropicalization of 
    $Z(f) \hookrightarrow (\mathbb{C}^2,0)$ if and only if 
     the initial form $\mathrm{in}_{(p,q)} f$ is not a monomial.
\end{corollary}

\begin{proof}
    By Corollary \ref{cor:arcweightplanecurve}, $(p,q)$ belongs to the local tropicalization of 
    $Z(f) \hookrightarrow (\mathbb{C}^2,0)$ if and only if $(p,q)$ is orthogonal to a compact edge of 
    the Newton polygon $\mathcal{N}(f)$, that is, if and only if the basis $\delta_{(p,q)}(f)$  
    has at least two elements. In turn, by Definition \ref{def:initform}, this is equivalent to the fact 
    that the initial form $\mathrm{in}_{(p,q)} f$ is not a monomial.
\end{proof}

{\em Corollary \ref{cor:reforminit} may be turned into a third definition of the local tropicalization 
of $Z(f) \hookrightarrow (\mathbb{C}^2,0)$, as the union of rays generated by the 
vectors $(p,q) \in  \mathbb{Z}_{>0}^{2}$ such that 
the initial form $\mathrm{in}_{(p,q)} f$ is not a monomial.} 
We will see below that a slight modification of this 
definition may be given for germs of higher dimension embedded in germs of toric varieties 
(see Definition \ref{def:initweightv}) and that 
it is equivalent to the two definitions presented in the introduction (see Theorem \ref{thm:coincidthm}). 
Before that, we need to explain our notations from toric geometry and valuation theory 
in arbitrary dimensions.

\medskip
\section{Basic facts about toric geometry and valuation theory}  \label{sect:notations}
\medskip

In this section we explain elementary facts as well as our notations and conventions 
about {\em toric geometry} and {\em valuation theory}. A basic relation between the 
two is that any {\em subtoric arc} in the sense of Definition \ref{def:arcs} defines a 
{\em semivaluation} (see Example \ref{ex:arcsemival}). 
\medskip

{\em Let us start with toric geometry.} Here we work in arbitrary dimension, by contrast with 
Sections \ref{sect:motivex} and \ref{sect:Newtonpolyg}, in which we 
considered only toric surfaces. The readers needing more details may consult the 
textbooks \cite{O 88}, \cite{F 93}, \cite{CLS 11}. Those who want quick introductions to 
the subject may consult \cite{B 01}, \cite{C 03} or \cite[Section 1.3]{GGP 20}. 

From now on, we replace the field $\mathbb{C}$ of the previous sections with an arbitrary algebraically 
closed field $\boxed{\mathbb{K}}$. Let $\boxed{N}$ and $\boxed{M}$ be two 
{\bf lattices}, \index{Lattice} that is, two free abelian groups of finite rank. 
We assume that they are {\bf dual} of each other, which means that we fix
a unimodular bilinear pairing 
     \[ \begin{array}{ccc}
           N \times M   &   \to   & \mathbb{Z}  \\
              (w, m) &  \mapsto &  \boxed{w \cdot m}
         \end{array} . \] 
 This pairing produces identifications
  \[ M =   \mbox{Hom}(N, \mathbb{Z}), \  N =   \mbox{Hom}(M, \mathbb{Z}).  \]
 We interpret the two lattices differently: $M$ is the {\bf lattice of exponents of monomials} 
 \index{Lattice!of exponents of monomials}
 and $N$ is the corresponding {\bf weight lattice}. \index{Lattice!weight} 
 That is, each vector $m \in M$ gives rise 
 to a {\bf monomial} $\boxed{\chi^m} \in \mathbb{K} [M]$, where $\boxed{\mathbb{K}[M]}$ denotes the 
 $\mathbb{K}$-algebra of the lattice $M$, and each vector $w \in N$ attributes the {\bf weight} 
 $w \cdot m \in \mathbb{Z}$ to the monomial  $\chi^m$. 
 
 Once a basis $(e_1, \dots, e_n)$ of $M$ is fixed, one gets a dual basis 
 $(e_1^{\vee}, \dots, e_n^{\vee})$ of $N$. One may write then   
 $m = m(1) e_1 + \cdots + m(n) e_n$ and $w = w(1) e_1^{\vee} + \cdots + w(n) e_n^{\vee}$, 
 where all the coefficients $m(i)$ and $w(i)$ are integers. 
 This allows to write $\chi^m = x_1^{m(1)} \cdots x_n^{m(n)}$, where 
 $\boxed{x_j} := \chi^{e_j}$ for every $j \in \{1, \dots, n\}$ and $w$ endows the variable 
 $x_j$ with the weight $w(j)$.  
 That is, $\mathbb{K}[M]$ is an intrinsic way of writing the $\mathbb{K}$-algebra of {\bf Laurent polynomials} \index{Polynomial!Laurent} in $n$ variables and $N$ is an intrinsic way of looking 
 at all the ways to endow the variables of such polynomials with integer weights.  
 
 Let $\tau$ be a {\bf strictly convex rational polyhedral cone} inside the real vector space 
 $\boxed{N_{\mathbb{R}}} :=  N \otimes_{\mathbb{Z}} \mathbb{R}$ generated by $N$.  
 That is, $\tau$ is a cone 
 $\boxed{\mathbb{R}_{\geq 0} \langle w_1, \dots, w_k \rangle}$ generated by a finite set of 
 vectors $w_j \in N$ and it is assumed not to contain vector subspaces of positive dimension. 
 We denote by $\boxed{\tau^{\circ}}$  the interior of $\tau$ in the 
 topological space $N_{\mathbb{R}}$ and by 
    \[  \boxed{\tau^{\vee}} := \{ m \in M_{\mathbb{R}}, \   w \cdot m \geq 0 \   \forall \  w \in \tau \} \subseteq M_{\mathbb{R}} \] 
 its {\bf dual cone}. \index{Dual cone} 
 The intersection $\tau^{\vee} \cap M$ is a submonoid of $(M, +)$, generating  
 a subalgebra $\mathbb{K}[\tau^{\vee} \cap M]$ of $\mathbb{K}[M]$. Its spectrum is by definition the 
 {\bf affine toric variety $\mathcal{X}_{\tau}$ \index{Toric!affine variety} 
 defined by the cone $\tau$ over the field $\mathbb{K}$}:  
      \[  \boxed{\mathcal{X}_{\tau}} :=  \mathrm{Spec} (\mathbb{K}[\tau^{\vee} \cap M]). \] 
      
   If $\tau$ is the origin $0$ of $N_{\mathbb{R}}$, one gets therefore 
   $\mathcal{X}_0 =  \mathrm{Spec} (\mathbb{K}[M])$. 
   This space is also denoted by $\boxed{T_N}$ and is called the {\bf algebraic torus} 
   \index{Torus!algebraic} defined over $\mathbb{K}$ 
   and corresponding  to the weight lattice $N$.

 The set $\boxed{\mathcal{X}_{\tau}(\mathbb{K})}$ of {\bf $\mathbb{K}$-valued points of $\mathcal{X}_{\tau}$} is identified naturally with 
   \[  \mbox{Hom}((\tau^{\vee} \cap M, +), (\mathbb{K}, \cdot)), \]
 where the morphisms are taken in the category of morphisms of monoids. This identification is obtained by associating to each $\mathbb{K}$-valued point $p \in \mathcal{X}_{\tau}(\mathbb{K})$ the evaluation map which sends every 
 $m \in \tau^{\vee} \cap M$ to $\chi^m(p) \in \mathbb{K}$. A basis $(e_1, \dots, e_n)$ of $M$ allows to identify the set $T_N(\mathbb{K})$ of $\mathbb{K}$-valued points of $T_N$ with $(\mathbb{K}^*)^n$.
  
 The variety $\mathcal{X}_{\tau}$ is {\em smooth} if and only if the cone $\tau$ is 
 \index{Cone!regular}{\bf regular}, that is, $\tau = \mathbb{R}_{\geq 0} \langle w_1, \dots, w_k \rangle$, where $(w_1, \dots, w_k)$ may be extended to a basis of the lattice $N$.  A choice of such an extension identifies $\mathcal{X}_{\tau}(\mathbb{K})$ with $\mathbb{K}^k \times (\mathbb{K}^*)^{n-k}$. Particular cases of this fact were listed as points \eqref{trivcone}, \eqref{raycone} and \eqref{reg2dcone} of Section \ref{sect:motivex}. 
 
 A {\bf fan} \index{Fan} 
 is a finite set $\mathcal{F}$ of strictly convex rational polyhedral subcones of the weight vector space 
 $N_{\mathbb{R}}$, satisfying the following properties:
    \begin{itemize}
        \item if $\tau \in \mathcal{F}$, then  all the faces of $\tau$ also belong to $\mathcal{F}$;
        \item if $\tau_1, \tau_2  \in \mathcal{F}$, then $\tau_1 \cap \tau_2$ is a face of both $\tau_1$ and $\tau_2$. 
    \end{itemize}
  The {\bf support} \index{Support!of a fan} $\boxed{| \mathcal{F} |} \subseteq N_{\mathbb{R}}$ 
 of the fan $\mathcal{F}$ is the union of its cones. All the fans of Figure \ref{fig:fansforcusp} have the same support, equal to the non-negative quadrant $\sigma_0$ of $\mathbb{R}^2$. 
 
 Every fan $\mathcal{F}$ determines a toric variety $\boxed{\mathcal{X}_{\mathcal{F}}}$, obtained by gluing the affine toric varieties $\mathcal{X}_{\tau}$ defined by the cones $\tau$ of $\mathcal{F}$. Each variety
 $\mathcal{X}_{\tau}$ becomes a Zariski-open subvariety of $\mathcal{X}_{\mathcal{F}}$. The torus $T_N = \mathcal{X}_0$ acts on $\mathcal{X}_{\mathcal{F}}$ by an extension of the canonical translational action on itself. This action partitions $\mathcal{X}_{\mathcal{F}}$ into orbits. Each open subset $\mathcal{X}_{\tau}$, for $\tau \in \mathcal{F}$, contains a unique orbit $\boxed{O_{\tau}}$ 
 \index{Orbit!of a toric variety} which is closed in $\mathcal{X}_\tau$. It is also its unique 
 smallest-dimensional orbit, and one has the equality:
   \[ \dim (O_{\tau}) + \dim (\tau) = \mbox{rk} (N). \]
 This generalizes point \eqref{inclpresbij} of Section \ref{sect:motivex}. 
 The {\bf toric boundary} \index{Toric!boundary} 
 $\boxed{\partial \mathcal{X}_{\tau}}$ of $\mathcal{X}_{\tau}$ is the complement of the dense torus of 
 $\mathcal{X}_{0}$.  That is, it is the union of orbits of codimension at least one  in $\mathcal{X}_\tau$. 
  
 If $| \mathcal{F} | \subseteq \sigma$, then there is a canonical toric birational morphism
    \[  \boxed{\pi_{\mathcal{F}}} :  \mathcal{X}_{\mathcal{F}} \to \mathcal{X}_{\sigma}  \]
  which identifies the dense tori of $\mathcal{X}_{\mathcal{F}}$ and  $\mathcal{X}_{\sigma}$. 
  This morphism is a modification in the sense of Definition~\ref{def:modif} 
  (that is, it is moreover {\em proper}) if and only  if $| \mathcal{F} | = \sigma$. 

\medskip
{\bf In the rest of the paper, we assume that $\boxed{\sigma} \subset N_{\mathbb{R}}$ 
denotes a strictly convex rational polyhedral cone with non-empty interior $\sigma^{\circ}$}. 
This means that the closed orbit $O_{\sigma}$ of $\mathcal{X}_{\sigma}$ has 
dimension zero. For simplicity, 
we denote it by $\boxed{0}$ and the monoid $(\sigma^{\vee} \cap M, +)$ by $\boxed{\Gamma}$.  
Let  $\boxed{(\mathcal{X}_{\sigma}, 0)}$ be the formal germ of $\mathcal{X}_{\sigma}$ at its closed orbit $0$. 
Its local ring $\boxed{\mathcal{O}_{\mathcal{X}_{\sigma}, 0}}$ is equal to the $\mathbb{K}$-algebra $\boxed{\mathbb{K}[[\Gamma]]}$ 
of formal power series with coefficients in $\mathbb{K}$ and exponents in the monoid 
$\Gamma$. If $f \in \mathbb{K}[[\Gamma]]$, its {\bf support} $\boxed{\mathrm{supp}\  f} \subseteq \Gamma$ may be defined as in the two-dimensional regular case 
(see Definition \ref{def:supptwod}). 

In Definition \ref{def:loctrop} we will explain what does it mean to tropicalize locally 
{\em interior subtoric germs}, in the following sense:

\begin{definition}   \label{def:subtorgerm}
     We say that a formal subgerm $Y \hookrightarrow (\mathcal{X}_{\sigma}, 0)$ is a 
     {\bf subtoric germ}.  \index{Subtoric!germ}
     If it is reduced and none of its irreducible components is contained in the toric boundary 
     $\partial \mathcal{X}_{\sigma}$, we call it an {\bf interior subtoric germ}. 
     \index{Subtoric!germ!interior}
 \end{definition}
 
 \begin{remark}  \label{rem:contmon}
      Denote by $\boxed{I_Y} \subseteq \mathbb{K}[[\Gamma]]$ the defining ideal of the subtoric germ $Y \hookrightarrow (\mathcal{X}_{\sigma}, 0)$. Note that $Y$ is contained in the toric boundary  $\partial \mathcal{X}_{\sigma}$ if and only if $I_Y$ contains monomials $\chi^m$, with $m \in \Gamma$. Therefore, an irreducible subtoric germ $Y \hookrightarrow (\mathcal{X}_{\sigma}, 0)$ is interior if and only if its defining ideal $I_Y$ does not contain monomials. 
   \end{remark}
 
Note that already in Section \ref{sect:intro} we considered only {\em interior} subtoric 
germs of $(\mathbb{C}^n, 0)$, 
without using this terminology. The plane curve singularities of Section \ref{sect:Newtonpolyg} 
were also assumed to be interior subtoric. 
Let us introduce an analogous terminology for {\em arcs}:

\begin{definition}  \label{def:arcs}
   An {\bf arc of a formal germ} $Y$  \index{Arc} is a formal morphism of the form 
         \[   \mbox{Spec}(\mathbb{K}[[t]]) \to Y.  \]
    A {\bf subtoric arc} \index{Arc!interior}  \index{Arc!interior!subtoric}
    is an arc of $(\mathcal{X}_{\sigma}, 0)$. An {\bf interior subtoric arc} is a subtoric arc 
    of $(\mathcal{X}_{\sigma}, 0)$ whose  image is not included in the toric boundary $\partial \mathcal{X}_{\sigma}$. 
 \end{definition}

 \medskip
We want to explain now how each arc of $Y$ defines a {\em semivaluation} of its local ring 
$\mathcal{O}_{Y,0}$ (see Example \ref{ex:arcsemival}). We first explain basic facts about 
valuations and semivaluations. In this paper, we always assume that the {\em semivaluations} are real-valued: 

\begin{definition}   \label{def:semival}  \index{Valuation}   \index{Semivaluation}
   Let $R$ be a commutative ring with $1 \neq 0$. A {\bf valuation centered in $R$} 
    is a map $\nu\colon R\to [0, +\infty]$ such that:
  \begin{itemize}
      \item $\nu(x)=+\infty$ if and only if $x=0$;
      \item $\nu(1) =0$;
      \item $\nu(xy)=\nu(x)+\nu(y)$  for all $x,y\in R$;
      \item $\nu(x+y)\geq \min\{\nu(x),\nu(y)\}$  for all $x,y\in R$.
  \end{itemize}
     A {\bf semivaluation centered in $R$} is defined similarly, 
     excepted that the first condition is relaxed to $\nu(0)=+\infty$. 
     That is, $\nu$ is allowed to take the value $+\infty$ on elements different from $0$. 
\end{definition}

For instance, the map 
   \[ \boxed{\nu_t} : \mathbb{K}[[t]] \to [0, + \infty] \]
associating to each non-zero formal power series its  initial exponent and to $0$ the symbol $+ \infty$ is a valuation of the $\mathbb{K}$-algebra $\mathbb{K}[[t]]$, called the {\bf $t$-adic valuation}.  
\index{Valuation!adic}

\begin{remark}   \label{rem:valconsid}
    The classical notion of {\em Krull valuation} has a field instead of a ring as source 
   and an arbitrary totally ordered abelian group enriched with the symbol $+ \infty$ 
   as target  (see \cite[Chapter VI.8]{ZS 60}). 
   This notion was later extended to rings, replacing fields by rings as sources 
   (see for instance  \cite[Section 2]{V 00}) and the notion of {\em semivaluation} was obtained 
   by allowing non-zero elements to take the value $+ \infty$. 
   In particular, real-valued semivaluations of a ring $R$ take values in $\mathbb{R} \cup \{+ \infty\}$. 
   Instead, the {\em centered} semivaluations of Definition \ref{def:semival} take only non-negative 
   values. {\bf In this paper we will only consider centered semivaluations, that is why we 
   suppress the qualificative {\em centered}.}  A centered semivaluation 
   has a {\em center} in $\mathrm{Spec} (R)$, defined by the prime ideal of elements $x \in R$ 
   such that $\nu(x) > 0$. We will not use this notion. 
\end{remark}

 It follows immediately from Definition \ref{def:semival} that if $\nu$ is a semivaluation on $R$, 
 then the set
     \[  \boxed{\mathfrak{p}_{\nu}} := \{x\in R \,|\, \nu(x)=+\infty\}   \]
is a prime ideal of $R$. 

\begin{definition}  \label{def:suppval}
      Let $\nu$ be a semivaluation of the ring $R$. We say that the closed subscheme of 
      $\mathrm{Spec}(R)$ defined by the prime ideal $\mathfrak{p}_{\nu}$ is the {\bf support} of $\nu$. 
      \index{Support!of a semivaluation}
\end{definition}

The name of {\em support} alludes to the fact that the semivaluation $\nu$ centered in $R$
is determined by a valuation centered in $R/\mathfrak{p}_{\nu}$. Indeed, the semivaluation $\nu$ 
descends to a valuation $\nu'$ of $R/\mathfrak{p}_{\nu}$, therefore 
it may be expressed as the composition:
      \[   R\to R/\mathfrak{p}_{\nu} \overset{\nu'}{\to} [0, +\infty].     \]
  More generally, whenever $\varphi : R \to R'$ is a morphism of rings and 
  $\nu' : R' \to [0, +\infty]$ is a valuation, the following composition is a semivaluation:
      \[   R  \overset{\varphi}{\to} R' \overset{\nu'}{\to} [0, +\infty].    \]

 \begin{example}   \label{ex:arcsemival}
          For instance, from the algebraic viewpoint an arc of $Y$ is defined by a morphism of rings 
   $ \varphi:  \mathcal{O}_{Y,0} \to  \mathbb{K}[[t]]$, therefore the $t$-adic valuation 
   $ \nu_t : \mathbb{K}[[t]] \to [0, +\infty]$ induces a semivaluation:
        \[  \nu_t \circ \varphi :  \mathcal{O}_{Y,0} \to [0, +\infty]. \]
  This is the announced explanation of the way an arc determines a semivaluation. Note that 
  this semivaluation is never a valuation when $Y$ is of dimension at least two, 
  because then there are non-zero elements $h$ of $\mathcal{O}_{Y,0}$ satisfying 
  $\varphi(h) = 0$.
 \end{example}

In Definitions \ref{def:subtorgerm} and \ref{def:arcs} we introduced the notions of 
{\em interior subtoric germs} and {\em interior arcs} respectively. Let us define an analogous 
notion for {\em semivaluations}, in the sense of Definition \ref{def:semival}: 

\begin{definition}   \label{def:intval}  \index{Semivaluation!interior}
     Let $\nu: \mathbb{K}[[\Gamma]] \to [0, + \infty]$ be a semivaluation. We call it an 
     {\bf interior semivaluation} if its support in the sense of Definition \ref{def:suppval} 
     is not contained in the toric boundary of $\mathcal{X}_{\sigma}$. 
     
     A semivaluation $\nu$ of the ring $\mathcal{O}_{Y,0}$ of a subtoric germ 
     $Y \hookrightarrow (\mathcal{X}_{\sigma},0)$ is called {\bf interior}, 
     if the semivaluation of $\mathbb{K}[[\Gamma]]$ obtained by composing $\nu$ 
     with the canonical projection $\mathbb{K}[[\Gamma]] \to \mathcal{O}_{Y,0}$ is so.
\end{definition}

The reader can check that a semivaluation $\nu : \mathbb{K}[[\Gamma]] \to [0, + \infty]$
is interior if and only if 
$\nu(\chi^m)\ne +\infty$ for all $m\in \Gamma$. Moreover, the semivaluation associated 
to an arc of a subtoric germ in the sense of Example \ref{ex:arcsemival} is interior if and 
only if the arc is itself interior.

\medskip
\section{Four sets of weight vectors associated to an interior subtoric germ}  \label{sect:weightsofarcsandvals}
\medskip

In this section we generalize the ideas of Section \ref{sect:Newtonpolyg} from 
plane curve singularities without branches contained in the coordinate axes to
arbitrary interior subtoric germs in the sense of Definition \ref{def:subtorgerm}. Namely,  
to every interior subtoric germ $Y \hookrightarrow (\mathcal{X}_{\sigma},0)$ 
we associate four kinds of 
vectors contained in the weight cone $\sigma$ (see Definitions 
\ref{def:compactifcone}, \ref{def:arcweights}, \ref{def:initweightv}, \ref{def:valweight}), 
which will be used in Theorem \ref{thm:coincidthm} to give four different characterizations of 
a conic subset of $\sigma$. It is this subset which we call 
the local tropicalization of $Y \hookrightarrow (\mathcal{X}_{\sigma},0)$ 
(see Definition \ref{def:loctrop}). 
\medskip

First, let us turn the toric-geometric explanations of Section \ref{sect:intro} into the following definition:

\begin{definition}   \label{def:compactifcone}
     Let $Y \hookrightarrow (\mathcal{X}_{\sigma},0)$ be an interior subtoric germ, in the sense 
     of Definition~\ref{def:subtorgerm}. \index{Fan!compactifying}
     A {\bf compactifying fan} of $Y$ is a fan $\mathcal{F}$  such that 
     $| \mathcal{F} | \subseteq \sigma$ 
     and such that the restriction $\pi : Y_{\mathcal{F}} \to Y$ 
     of the toric birational morphism 
     $\pi_{\mathcal{F}} : \mathcal{X}_{\mathcal{F}} \to \mathcal{X}_{\sigma}$ to the strict transform 
     $\boxed{Y_{\mathcal{F}}}$ of $Y$ by $\pi_{\mathcal{F}}$ is {\em proper}. 
     The {\bf compactifying cone} of $Y$ \index{Cone!compactifying}
     is the intersection of the supports of all its compactifying fans. 
\end{definition}

\begin{example}   \label{ex:compactfan}
     Let us give a simple example of fan $\mathcal{F}$ such that $\pi_{\mathcal{F}}$ is not proper but 
     $\pi$ is proper. Choose $n=2$ and let $Y$ be the line $Z(y-x)$ 
     of the  affine plane $\mathbb{K}^2_{x,y}$ with coordinates $x,y$.  
     Let $ \pi_{\mathcal{F}}$ be the toric birational morphism $\mathbb{K}^2_{u,v} \to 
     \mathbb{K}^2_{x,y}$ defined 
     by the equations $x = u, y = uv$, that is, one of the two charts of the blowup of the origin 
     of $\mathbb{K}^2_{x,y}$. The corresponding fan $\mathcal{F}$ consists of the faces of the cone 
     $\mathbb{R}_{\geq 0}\langle  (1,1), (0,1) \rangle$ 
     spanned by the vectors $(1,1)$ and $(0,1)$. The strict transform of $Y$ by $\pi_{\mathcal{F}}$ is 
     $Y_{\mathcal{F}} = Z(v -1)$. 
     The toric morphism $\pi_{\mathcal{F}}$ is obviously not proper, as the preimage 
     $\pi_{\mathcal{F}}^{-1}(0) = Z(u) \hookrightarrow \mathbb{K}^2_{u,v}$ of the origin of 
     $\mathbb{K}^2_{x,y}$ is not complete. 
     But the restriction $\pi : Z(v-1) \to Z(y-x)$ of $\pi_{\mathcal{F}}$ to $Y_{\mathcal{F}}$ 
     is an isomorphism, hence it is proper. 
\end{example}

  Let us look at the compactifying cone of an interior subtoric branch 
  $B \hookrightarrow (\mathcal{X}_{\sigma},0)$, in the sense of Definition \ref{def:compactifcone}. 
  A normalisation morphism $\mathrm{Spec} (\mathbb{K}[[t]]) \to B$ is an interior subtoric arc 
  $\gamma : \mathrm{Spec}(\mathbb{K}[[t]]) \to (\mathcal{X}_{\sigma},0)$,  in the sense of 
  Definition \ref{def:arcs}. 
  Each such arc has an associated weight vector $w(\gamma)$, defined as follows. Denote by  
      \[ \boxed{\gamma^*} : (\mathbb{K}[[\Gamma]], \cdot)  \to (\mathbb{K}[[t]], \cdot).  \] 
   the corresponding local morphism of $\mathbb{K}$-algebras seen
   as a morphism of multiplicative monoids.  
    By composing $\gamma^*$ at the source with the natural injection of monoids 
       $\boxed{i_{\Gamma}} : (\Gamma, +) \hookrightarrow (\mathbb{K}[[\Gamma]], \cdot)$ 
       and at the target with the 
       $t$-adic valuation seen as a 
       morphism $ \nu_t : (\mathbb{K}[[t]], \cdot)  \to (\mathbb{Z}_{\geq 0} \cup \{+ \infty\}, +) $ of monoids, 
       we get a morphism of monoids 
           $  \nu_t \circ \gamma^* \circ i_{\Gamma}  : (\Gamma, +) \to (\mathbb{Z}_{\geq 0}, +)$,  
        that is, an element of the monoid $(\sigma \ \cap \ N, +)$. 
       As the arc $\gamma$ is assumed to be subtoric, one has 
       $ (\nu_t \circ \gamma^* \circ i_{\Gamma})(\Gamma \smallsetminus \{0\}) \subseteq \mathbb{Z}_{>0}$, 
       which shows that $\nu_t \circ \gamma^* \circ i_{\Gamma} \in \sigma^{\circ}  \cap   N$. 
       We may formulate:

 \begin{definition}   \label{def:weightarc}
       Let   $\gamma : \mathrm{Spec}(\mathbb{K}[[t]]) \to (\mathcal{X}_{\sigma},0)$ 
       be an interior subtoric arc. \index{Weight vector!of an interior subtoric arc}
       Its {\bf weight vector $\boxed{w(\gamma)}$} is the morphism of monoids 
            \[  \nu_t \circ \gamma^* \circ i_{\Gamma}  : (\Gamma, +) \to (\mathbb{Z}_{\geq 0}, +), \]
        seen as an element of $\sigma^{\circ}  \cap   N$.            
 \end{definition}
 
For instance, if $\mathcal{X}_{\sigma} (\mathbb{K})= \mathbb{K}^2$ and $\gamma$ is defined by 
$t \mapsto (x(t), y(t))$ as in formula \eqref{eq:expofxy}, one gets  
$w(\gamma) = (p,q) = (\nu_t(x(t)), \nu_t(y(t)))$.

 One has the following property, which shows that {\em in the case of interior subtoric branches}, 
 the two definitions of local tropicalization given in Section \ref{sect:intro} are equivalent: 
 
 \begin{proposition}   \label{prop:equivintbr}
     Let $B \hookrightarrow (\mathcal{X}_{\sigma},0)$ be an interior subtoric branch and let 
     $\gamma : \mathrm{Spec}(\mathbb{K}[[t]]) \to (\mathcal{X}_{\sigma}, 0)$ be an arc 
     whose image is $B$. Then, the ray 
     $\mathbb{R}_{\geq 0} w(\gamma) \subset \sigma$ generated by its weight vector in the sense of 
     Definition \ref{def:weightarc} is independent 
     of the choice of the arc $\gamma$ and is equal to the 
     compactifying cone of $B$, in the sense of Definition \ref{def:compactifcone}. 
 \end{proposition}
 
 \begin{proof}
     Consider the affine toric variety $\mathcal{X}_{\mathbb{R}_{\geq 0} w(\gamma)}$. 
     One may check that the arc 
     $\gamma$ may be lifted to an arc $\gamma' : \mathrm{Spec}(\mathbb{K}[[t]]) \to 
            \mathcal{X}_{\mathbb{R}_{\geq 0} w(\gamma)}$ 
     which sends the origin of $\mathrm{Spec}(\mathbb{K}[[t]])$ to a point of the closed orbit 
     $O_{\mathbb{R}_{\geq 0} w(\gamma)}$ of $\mathcal{X}_{\mathbb{R}_{\geq 0} w(\gamma)}$. 
     This proves that the restriction $\pi' : B' \to B$ of the toric morphism 
     $\pi_{\mathbb{R}_{\geq 0} w(\gamma)} : \mathcal{X}_{\mathbb{R}_{\geq 0} 
            w(\gamma)} \to \mathcal{X}_{\sigma}$ 
     to the strict transform $B'$ of $B$ is proper. As a consequence, the restriction 
     $\pi : B_{\mathcal{F}} \to B$ of $\pi_{\mathcal{F}}$ to the strict transform $B_{\mathcal{F}}$ 
     of $B$ by $\pi_{\mathcal{F}}$ 
     is proper whenever $\mathbb{R}_{\geq 0} w(\gamma) \subseteq | \mathcal{F} |$. 
     
     Indeed, consider a subdivision $\mathcal{F}'$ of $\mathcal{F}$ containing 
     $\mathbb{R}_{\geq 0} w(\gamma)$ among its rays. Then $\mathcal{X}_{\mathcal{F}'}$ 
     contains the affine toric variety $\mathcal{X}_{\mathbb{R}_{\geq 0} w(\gamma)}$ 
     as a Zariski-open subset.  
     The restriction of the toric morphism $\pi_{\mathcal{F}'}$ to this subset is equal to
     $\pi_{\mathbb{R}_{\geq 0} w(\gamma)} : \mathcal{X}_{\mathbb{R}_{\geq 0} 
           w(\gamma)} \to \mathcal{X}_{\sigma}$. 
     This shows that the restriction of $\pi_{\mathcal{F}'}$ to the strict transform $B_{\mathcal{F}'}$ of 
     $B$ on $\mathcal{X}_{\mathcal{F}'}$ is proper. As $\pi_{\mathcal{F}'}$ factors 
     through $\pi_{\mathcal{F}}$, 
     this implies that $\pi : B_{\mathcal{F}} \to B$ is also proper. 
 \end{proof}
 
 Let us turn now to arbitrary interior subtoric germs. As explained in Section \ref{sect:intro}, 
 we may look at all the interior subtoric arcs contained in them and at the corresponding weight vectors: 

\begin{definition}   \label{def:arcweights}
    Let $Y \hookrightarrow (\mathcal{X}_{\sigma},0)$ be an interior subtoric germ, in the sense of    
    Definition \ref{def:subtorgerm}. An {\bf arcwise weight vector} \index{Weight vector!arcwise} 
    of $Y$ is the weight vector of an interior subtoric arc contained in $Y$, 
    in the sense of Definition \ref{def:weightarc}. 
\end{definition}

Instead of looking at the arcs $(\mathrm{Spec} (\mathbb{K}[[t]]),0) \to (\mathcal{X}_{\sigma},0)$ 
whose image is contained in $Y$, we may look dually at the formal germs of functions 
on $(\mathcal{X}_{\sigma}, 0)$ vanishing on $Y$, 
that is, at the elements of the ideal $I_Y \subseteq \mathbb{K}[[\Gamma]]$ defining $Y$ in 
$(\mathcal{X}_{\sigma}, 0)$ (see Remark \ref{rem:contmon}). Then, one may again define preferred weight 
vectors (see Definition \ref{def:initweightv}). Note first that Definition \ref{def:initform} readily generalizes 
to an arbitrary weight vector $w\in \sigma^\circ$
and a formal series $f\in \mathbb{K}[[\Gamma]]$, respectively an ideal $I$ of $\mathbb{K}[[\Gamma]]$:

\begin{definition}   \label{def:initoric}
    Consider $w\in \sigma^\circ$  and $\displaystyle{f = \sum_{m\in \  \mathrm{supp}\  f} 
           f_m\, \chi^m \in \mathbb{K}[[\Gamma]]}$.   \index{Basis!of a support}
    The {\bf basis $\boxed{\delta_w(f)} \subseteq \mathrm{supp}\  f$ 
    of $\mathrm{supp}\  f$ relative to $w$} is the locus where the restriction  
    of the linear form $w: M_{\mathbb{R}} \to \mathbb{R}$ to the support $\mathrm{supp}\  f$ 
    of $f$ achieves its minimum. The $w$-weighted homogeneous polynomial
      \[ \boxed{\mathrm{in}_{w} f} :=\sum_{m\in \ \delta_w(f)} f_m \, \chi^m   \]
   is called the {\bf $w$-initial form} of the series $f$.
   
    If $I\hookrightarrow \mathbb{K}[[\Gamma]]$ is an ideal, then its
    {\bf $w$-initial ideal} $\boxed{\mathrm{in}_w I}$ is the ideal of $\mathbb{K}[[\Gamma]]$ 
     generated by the $w$-initial  forms $\mathrm{in}_w f$, for all $f\in I$.
     \index{Initial!form of a series}   \index{Initial!ideal}
\end{definition}

Let us distinguish the weight vectors which define initial ideals without monomials 
(recall Remark \ref{rem:contmon}), generalizing the two-dimensional situation 
considered in Corollary \ref{cor:reforminit}:

\begin{definition}   \label{def:initweightv}
    Let $Y \hookrightarrow (\mathcal{X}_{\sigma},0)$ be an interior subtoric germ and 
    $I_Y \hookrightarrow \mathbb{K}[[\Gamma]]$ \index{Initial!weight vector}
    be its defining ideal. A weight vector $w \in \sigma^{\circ}$ is called an {\bf initial weight 
    vector of $Y$} if the initial ideal $\mathrm{in}_w I_Y \hookrightarrow \mathbb{K}[[\Gamma]]$ 
    does not contain monomials. 
\end{definition}

Finally, let us also define weight vectors associated to semivaluations: 

\begin{definition}   \label{def:valweight}
    Let $Y \hookrightarrow (\mathcal{X}_{\sigma},0)$ be an interior subtoric germ. 
    Let  $\nu$ be an interior  semivaluation of the ring $\mathcal{O}_{Y,0}$, 
    in the sense of Definition \ref{def:intval}. 
    We look at it as a morphism of monoids $\nu : (\mathcal{O}_Y, \cdot) \to ([0, + \infty], +)$. 
    Denote by $\boxed{r_Y} : (\Gamma, +) \to  (\mathcal{O}_Y, \cdot)$ the composition of the 
    morphism $i_{\Gamma} : (\Gamma, +) \hookrightarrow (\mathbb{K}[[\Gamma]], \cdot)$ 
    of Definition \ref{def:weightarc} and of the restriction morphism 
    $(\mathbb{K}[[\Gamma]], \cdot) \to (\mathcal{O}_Y, \cdot)$. By composing $\nu$ and $r_Y$ 
    we get a morphism of monoids 
        \[   \nu \circ r_Y : (\Gamma, +) \to  ([0, + \infty], +), \]
     that is, an element $\boxed{w(\nu)}$ of $\sigma$. 
       We call it the {\bf weight vector of the interior semivaluation $\nu$}.  
         \index{Weight vector!of an interior semivaluation}
       The weight vector of an interior semivaluation of $\nu$ is called a {\bf valuative weight 
       vector of $Y$}.   \index{Weight vector!valuative}
\end{definition}

One may check that:
  \begin{itemize}
     \item  as $\nu$ was assumed to be interior, one has $w(\nu) \in \sigma^{\circ}$;
     \item the weight vector of the semivaluation associated 
to an interior arc $\gamma$ of $Y$ as explained in Example \ref{ex:arcsemival} 
is equal to the weight vector $w(\gamma)$ of $\gamma$ in the sense of Definition \ref{def:weightarc}. 
   \end{itemize}

\medskip
\section{Four viewpoints on local tropicalization}  \label{sect:coincid}
\medskip

This section presents the main result of this paper, Theorem \ref{thm:coincidthm}, 
stating that given an interior subtoric germ, four subsets of the real weight space 
constructed using the notions introduced in Section \ref{sect:weightsofarcsandvals} coincide. 
This provides four different interpretations of the {\em local tropicalization} of an interior 
subtoric germ (see Definition \ref{def:loctrop}). We conclude the section with a structure 
result about local tropicalizations (see Theorem \ref{thm:structloctrop}).
\medskip

In the sequel we will use the following notion: 

\begin{definition}   \label{def:coneclosure}
   Let $\Sigma$ be a subset of a finite-dimensional real vector space $V$. Its {\bf cone-closure} 
   is the closure $\boxed{\overline{\mathbb{R}_{\geq 0} \Sigma}}$ inside $V$ of the non-negative cone 
   $\boxed{\mathbb{R}_{\geq 0} \Sigma}$ over $\Sigma$,   \index{Cone-closure}
   consisting of all the rays $\mathbb{R}_{\geq 0} w$ generated by the elements $w$ of $\Sigma$. 
 \end{definition}

The following result confirms that for each interior subtoric germ, the cone-closures of 
four different subsets of the weight cone of the ambient affine toric variety agree: 

\begin{theorem}   \label{thm:coincidthm}
    Let $Y \hookrightarrow (\mathcal{X}_{\sigma},0)$ be an interior subtoric germ, 
    in the sense of Definition \ref{def:subtorgerm}. Then, the following subsets of the 
    weight cone $\sigma \subset N_{\mathbb{R}}$ coincide: 
      \begin{enumerate}
           \item  \label{eq:compactcone}
               the compactifying cone of $Y$, in the sense of Definition \ref{def:compactifcone};
           \item \label{eq:closarcwise}
               the cone-closure of the set of arcwise weight vectors of $Y$, in the sense of 
                Definition \ref{def:arcweights}; 
           \item    \label{eq:closinitw}
                the cone-closure of the set of initial weight vectors of $Y$, in the sense of 
                Definition \ref{def:initweightv}; 
           \item    \label{eq:closvalweight}
                the cone-closure of the set of valuative weight vectors of $Y$, in the sense of 
                Definition \ref{def:valweight}. 
       \end{enumerate} 
\end{theorem}

\begin{proof} $\ $ 
   
   $\bullet$ When the field $\mathbb{K}$ has characteristic $0$,  the equality of the sets 
      defined in {\em \ref{eq:closarcwise}.} and {\em \ref{eq:closinitw}.} is a consequence of 
      \cite[Corollary 4.3]{S 17}.   Here we present a sketch of a different proof valid 
      in arbitrary characteristic. In fact, more is true: both starting cones agree, not only their 
      cone-closures. Let us explain this fact. Fix a ray $\lambda \subset \sigma$ 
      which is spanned by an element of $\sigma^{\circ} \cap N$. One proves that 
      each of the following statements is  equivalent to the next one:
      \begin{enumerate}[label=\alph*)]
         \item there exists an arcwise weight vector $w$ of $Y$ spanning $\lambda$;
         \item   \label{exsubtoric} 
              there exists an interior subtoric arc of $Y$ whose strict transform 
              on the affine toric variety $\mathcal{X}_{\lambda}$ intersects its unique 
              closed orbit $O_{\lambda}$; 
          \item    \label{nempty} 
               the strict transform $Y_{\lambda}$ of $Y$ on $\mathcal{X}_{\lambda}$ is such that 
               $Y_{\lambda} \cap O_{\lambda}$ is non-empty; 
          \item     \label{inclinit}
               $\lambda \smallsetminus \{0\}$ is included in the set of initial weight vectors of $Y$. 
      \end{enumerate}

 Therefore, the non-negative cone over the set of arcwise weight vectors of $Y$ coincides with the 
 non-negative cone over the set of initial weight vectors of $Y$. 
 
 The equivalence  \ref{exsubtoric} $\Longleftrightarrow$  \ref{nempty} uses the fact 
 that in any characteristic,  
 branches admit parametrizations, that is, are images of arcs (see 
 \cite[Section 4.1]{C 04} for an argument using their resolution by blowups and \cite[Section 1.3.2]{CC 05} 
 for an argument using the regularity of their normalizations).  
   
  The equivalence \ref{nempty} $\Longleftrightarrow$ \ref{inclinit} results from the fact that if 
  $w \in N$ spans the ray $\lambda$, then the initial ideal 
  $\mathrm{in}_w I_Y$ of the defining ideal $I_Y \subset \mathbb{K}[[\Gamma]]$ of $Y$, seen 
  as an ideal of $\mathbb{K}[M]$ instead of $\mathbb{K}[[\Gamma]]$ as in Definition \ref{def:initoric}, 
  defines a subvariety of the dense torus $T_N$ whose closure in $\mathcal{X}_{\lambda}$ 
  intersects  $O_{\lambda}$ at a point of the intersection $Y_{\lambda} \cap O_{\lambda}$.

     \medskip
     $\bullet$  The equality of the sets defined in {\em \ref{eq:compactcone}.} and 
        {\em \ref{eq:closinitw}.} is a consequence of \cite[Proposition 3.19]{CPS 24}. 
        Let us explain the main ideas of the proof, generalizing again the situation from
    the case $\mathbb{K}=\mathbb{C}$ of \cite{CPS 24} to the case of arbitrary field $\mathbb{K}$.
     Denote by $\mathcal{F}$ a fan whose support is contained 
     in $\sigma$ and by $\boxed{\mathfrak{W}(Y)}$
     the set of initial weight vectors of $Y$ in the sense of Definition \ref{def:initweightv}. Choose a
     fan $\Sigma$ subdividing the cone $\sigma$ and such that $\mathcal{F} \subseteq \Sigma$. 
     Then, we get the following commutative diagram
     \[   \xymatrix{
              Y_\Sigma \ar@{^{(}->}[r] & \mathcal{X}_{\Sigma} \ar[rd]_{\pi_\Sigma} & 
              & \mathcal{X}_{\mathcal{F}} \ar@{_{(}->}[ll] \ar[ld]^{\pi_{\mathcal{F}}} & 
                     Y_{\mathcal{F}} \ar@{_{(}->}[l] \ar[ld]^{\pi} \\
             & & \mathcal{X}_{\sigma} & Y \ar@{_{(}->}[l]}
      \]
     where $Y_{\mathcal{F}}$ and $Y_\Sigma$ are the strict transforms of 
     $Y$ in $\mathcal{X}_{\mathcal{F}}$ and $\mathcal{X}_\Sigma$
     respectively, $\pi_{\mathcal{F}}$ and $\pi_\Sigma$ are the natural toric birational morphisms, 
     and $\pi$ is
     the restriction of $\pi_{\mathcal{F}}$ to $Y_{\mathcal{F}}$. Note that $\pi$ is {\em proper} 
     (that is, $\mathcal{F}$ is a compactifying fan of $Y$) if and only if
     $Y_\Sigma$ is contained in $\mathcal{X}_{\mathcal{F}}$. 
     
     The equality of the sets defined in {\em \ref{eq:compactcone}.} and 
     {\em \ref{eq:closinitw}.} 
     is therefore a consequence of the following equivalence, for every
     cone $\tau$ of $\Sigma$:
     \begin{equation}   \label{eq:equivtoric}
          O_\tau \cap Y_\Sigma\neq \emptyset \Longleftrightarrow
                  \tau^{\circ}\cap \mathfrak{W}(Y)\neq \emptyset,
     \end{equation}
     where $O_\tau$ is the corresponding toric orbit of $\mathcal{X}_{\Sigma}$.

     {\bf Let us first prove the implication $\Longrightarrow$ of \eqref{eq:equivtoric}.}  
     Choose $y_0\in O_\tau\cap Y_\Sigma$.  
     Let $\gamma : \mathrm{Spec}(\mathbb{K}[[t]]) \to (Y_{\Sigma},y_0)$ 
     be an arc passing through $y_0$.  
     Denote by $w \in N$ the weight vector (in the sense of Definition \ref{def:weightarc}) 
     of the interior subtoric arc $\pi_{\Sigma} \circ \gamma : \mathrm{Spec}(\mathbb{K}[[t]]) \to Y$.    
     The fact that $y_0\in O_\tau\cap Y_\Sigma$ means that the weight vector $w$
     lies in $\tau^{\circ}\cap N \subseteq \tau^{\circ}$. To show that $w\in \mathfrak{W}(Y)$, we let  
     $B$ be the branch contained in $Y$ which is the image of the arc 
     $\pi_{\Sigma} \circ \gamma$. Since $B\subseteq Y$, the ideal $I_Y \subset \mathbb{K}[[\Gamma]$ 
     of $Y$ is contained in the ideal $I_B \subset \mathbb{K}[[\Gamma]$ of $B$. Thus, 
     $\mathfrak{W}(B)\subseteq \mathfrak{W}(Y)$. But, by arguments similar to those used 
     in the proof of Proposition \ref{prop:equivintbr}, we see that $w \in \mathfrak{W}(B)$.
     
     {\bf Let us now prove the implication $\Longleftarrow$ of \eqref{eq:equivtoric}.}  
     Consider a primitive lattice vector $w\in \tau^{\circ}\cap \mathfrak{W}(Y)$ and 
      a refinement $\Sigma_w$ of $\Sigma$ such that 
      the ray $\lambda_w :=\mathbb{R}_{\geq 0}\cdot w$ is a cone of the fan $\Sigma_w$. By construction, 
      the orbit $O_{\lambda_w}$ is mapped via the toric morphism $\pi_w \colon 
      \mathcal{X}_{\Sigma_w}\to \mathcal{X}_{\Sigma}$ to
      the orbit $O_\tau$. The intersection of the strict transform $Y_{\Sigma_w}$ of $Y$ in 
      $ \mathcal{X}_{\Sigma_w}$ with the orbit $O_{\lambda_w}$ is 
      determined by the $w$-initial ideal $\mathrm{in}_w {I(Y)}$ of the ideal $I(Y)$ defining $Y$, 
      viewed in the Laurent polynomial ring $\mathbb{K}[M]$. As $w\in \mathfrak{W}(Y)$, 
      this initial ideal is monomial free.  
      Therefore, one has $Y_{\Sigma_w} \cap O_{\lambda_w}  \neq \emptyset$. 
      Since $O_{\tau_w} \subset \mathcal{X}_{\Sigma_w}$, the map $\pi_w$ ensures that
      $Y_\Sigma\cap O_\tau \neq \emptyset$ as well.
    
    \medskip
   $\bullet$ The equality of the sets defined in {\em \ref{eq:closinitw}.} and 
   {\em \ref{eq:closvalweight}.} is a 
   consequence of \cite[Theorem 11.2]{PS 13}. In fact, the inclusion of the set of 
   {\em \ref{eq:closvalweight}.} 
   inside the set of {\em \ref{eq:closinitw}.} follows from essentially the same argument as described before
   Theorem \ref{thm:pqpair}. Namely, if $f\in I_Y \subset \mathbb{K}[[\Gamma]]$ 
   is a series of the defining ideal of $Y$ and 
   $\nu\colon \mathbb{K}[[\Gamma]]/I_Y \to \mathbb{R}\cup \{+\infty\}$ 
   is an interior semivaluation, then $\nu$ lifts to an interior semivaluation $\nu'$ of 
   $\mathbb{K}[[\Gamma]]$ 
   which necessarily takes value $+\infty$ on $f$. Since
   $\nu'$ does not take infinite value on any monomial 
   $\chi^m$ with $m \in \Gamma$, it follows that $f$ must involve
   at least two monomials where $\nu'$ attains its minimal value. Thus $\mathrm{in}_w(f)$, 
   where $w$ is the
   weight vector corresponding to $\nu'$, is not a monomial. The inverse inclusion is based on subtler
   results from the theory of valuations, see \cite[Theorem 11.2]{PS 13} and references therein. 
\end{proof}

Theorem \ref{thm:coincidthm} provides four equivalent definitions of the {\em local tropicalization} of an interior subtoric germ:  

\begin{definition}    \label{def:loctrop}
     Let $Y \hookrightarrow (\mathcal{X}_{\sigma},0)$ be an interior subtoric germ, 
     in the sense of Definition \ref{def:subtorgerm}. 
     The set defined by any of the four equivalent conditions of Theorem \ref{thm:coincidthm} is called
     the {\bf local tropicalization} of the subtoric germ $Y$ and is denoted by 
     $\boxed{\mathrm{Trop}_{\mathrm{loc}} Y} \subset \sigma$.  \index{Local tropicalization}
     \index{Tropicalization!local}
\end{definition}

The local tropicalization $\mathrm{Trop}_{\mathrm{loc}} Y$ depends strongly on the embedding  
$Y \hookrightarrow (\mathcal{X}_{\sigma},0)$ (think, for instance, about the various 
embeddings $t \mapsto (t, t^n)$ of a smooth branch inside $(\mathbb{K}^2,0)$, 
whose local tropicalizations are, by Proposition \ref{prop:equivintbr}, the rays spanned 
by $(1,n)$). Even if we drop it from the notation for simplicity, 
one has to keep this dependence  in mind. 

   Each of the four viewpoints on local tropicalization presented in Theorem \ref{thm:coincidthm} 
   is adapted to particular contexts. For instance:
   \begin{itemize}
      \item  viewpoint {\em \ref{eq:compactcone}.} is useful when one wants to construct modifications of 
           $Y$ starting from a particular embedding in a germ of toric variety; 
      \item viewpoint {\em \ref{eq:closarcwise}.} is useful when one wants to think in terms 
          of the arc space of the germ $Y$; 
      \item viewpoint {\em \ref{eq:closinitw}.} is useful when one deals with concrete defining 
          equations of $Y$, as in our work \cite{CPS 24} discussed in Section \ref{sec:splicetype}; 
      \item viewpoint {\em \ref{eq:closvalweight}.} is useful when one wants to think in terms of the space 
          of semivaluations associated to $Y$; for instance, our most general definitions of tropicalization 
          given in \cite[Section 6]{PS 13} used this space.
   \end{itemize}

   In fact, in our paper \cite{PS 13} we introduced two notions of local tropicalization, a 
   {\em positive} and a {\em nonnegative} one (see \cite[Definitions 6.6 and 6.7]{PS 13}), both being 
   {\em extended} local tropicalizations in the sense explained in Section \ref{sect:variants}. 
   For simplicity, we preferred to take here as definition of local tropicalization only the intersection of 
   the nonnegative local tropicalization of \cite{PS 13} with the cone $\sigma$.

   \begin{remark}
       Combining  Proposition \ref{prop:equivintbr} with Theorem \ref{thm:coincidthm}, we see 
       that the local tropicalization of an interior subtoric branch of $(\mathcal{X}_{\sigma}, 0)$ 
       is the ray spanned by 
       the weight vector of any arc parametrizing the branch. Maurer had introduced 
       such weight vectors in his 1980 paper \cite{M 80} for curve singularities 
       $Y \hookrightarrow (\mathbb{C}^n,0)$ with several branches, 
       under the name of {\em critical tropisms}. For this reason, we may 
       see Maurer's paper as a precursor of viewpoint  {\em \ref{eq:closarcwise}.} on local tropicalization. 
       The semantic proximity of the noun {\em tropism}   \index{Tropism!critical}  \index{Critical tropism}
       with the adjective {\em tropical}  is interesting. Lejeune-Jalabert and Teissier had 
       already used the expression {\em tropisme critique} in \cite{LT 73} 
       (see \cite[end of Section 1.4.5]{GGP 20}). 
   \end{remark}

   Viewpoint {\em \ref{eq:closarcwise}.} implies that the local tropicalization of the union 
   of two interior subtoric germs of $(\mathcal{X}_{\sigma},0)$ is the union of their two local tropicalizations. 
   Therefore, every local tropicalization is a finite union of local tropicalizations of irreducible 
   interior subtoric germs. Our main structure theorem about the local tropicalizations of such germs 
   becomes (see \cite[Theorem 11.9]{PS 13}, \cite[Proposition 3.11]{CPS 24}): 

\begin{theorem}   \label{thm:structloctrop}
     Let $Y \hookrightarrow (\mathcal{X}_{\sigma},0)$ be an irreducible interior subtoric germ 
     with defining ideal $I_Y  \subset   \mathbb{K}[[\Gamma]]$. Then, the local tropicalization 
     $\mathrm{Trop}_{\mathrm{loc}} Y$ is the support of a fan $\mathcal{F}$ which satisfies 
     the following conditions:
\begin{itemize}
   \item the dimension of each maximal cone of $\mathcal{F}$ is equal to the dimension of $Y$;
   \item the maximal cones of $\mathcal{F}$ all have non-empty intersections with $\sigma^\circ$;
   \item given any cone $\tau$ of $\mathcal{F}$, the $w$-initial ideal of $I_Y$ is independent 
       of the choice of $w\in\tau^\circ$.
\end{itemize}
\end{theorem}

Viewpoint {\em \ref{eq:closinitw}.} implies that:

\begin{proposition}
    The local tropicalization of an interior subtoric germ is equal to the intersection 
    of the local tropicalizations of the interior principal effective divisors containing it. 
\end{proposition}

This fact, which is a local version of one of the definitions of global tropicalization of 
a subvariety of an algebraic torus (see \cite[Definition 3.2.1]{MS 15}), 
explains the importance of understanding 
the local tropicalizations of effective principal  divisors on $(\mathcal{X}_{\sigma}, 0)$. 
Next section is dedicated to this topic.

\medskip
\section{The local tropicalization of an effective principal subtoric divisor}  \label{sect:localtropprinc}
\medskip

In this section we describe the local tropicalizations of the interior subtoric germs which are  
{\em effective principal divisors}. They are determined by the Newton polyhedra of their defining series 
in a way which generalizes the case of plane curve singularities described in Corollary 
\ref{cor:arcweightplanecurve} (see Theorem \ref{thm:divloctrop}). 
\medskip

We continue using the notations of Section \ref{sect:notations}. 
We assume that the interior subtoric germ $Y \hookrightarrow (\mathcal{X}_{\sigma},0)$ 
is an effective principal divisor on $\mathcal{X}_{\sigma}$. Therefore $Y = Z(f)$,
where $f\in \mathbb{K}[[\Gamma]]$. The fact that no irreducible component of $Y$ is contained in the
toric boundary $\partial \mathcal{X}_{\sigma}$ is equivalent to the condition that the series $f$ 
is not contained in any of the ideals
of $\mathbb{K}[[\Gamma]]$ generated by the sets of monomials 
$(\sigma^\vee \smallsetminus \tau^\perp) \cap M$, 
where $\tau$ is an edge of $\sigma$. For example, if $\sigma$ is the positive
orthant $\mathbb{R}_{\geq 0}^{n}$, which allows to write $\mathbb{K}[[\Gamma]]=\mathbb{K}[[x_1,\dots,x_n]]$, 
this simply says that $f$ is not divisible by any of the variables $x_1,\ldots,x_n$. 

The following definition is a straightforward generalization of Definition \ref{def:NP}:

\begin{definition}  \label{def:NPtoric}
   Let $\displaystyle{f =\sum_{m\in \ \mathrm{supp}\  f} f_m \chi^m \in \mathbb{K}[[\Gamma]]}$ 
   (therefore $f_m\in \mathbb{K}^*$ 
   for every $m \in \mathrm{supp}\  f$).  \index{Newton!polyhedron}
   Its  {\bf Newton polyhedron} $\boxed{\mathcal{N}(f)} \subseteq \sigma^{\vee}$ is 
   the convex hull of the set
          \[    \bigcup_{m\in \ \mathrm{supp}\  f} (m+\sigma^\vee) \subset \sigma^\vee.  \]
\end{definition}

Note that, as in the two-dimensional case examined in Section \ref{sect:Newtonpolyg}, 
for every $w \in \sigma^{\circ}$ the basis $\delta_w(f)$ appearing in Definition \ref{def:initoric} 
is the intersection of $\mathrm{supp}\  f$ with a compact face of $\mathcal{N}(f)$. 
This face is the locus where the linear form $w: M_{\mathbb{R}} \to \mathbb{R}$ 
achieves its minimum in restriction to the polyhedron $\mathcal{N}(f)$.

There is a duality between the faces of the Newton polyhedron $\mathcal{N}(f)$ and 
a certain fan with support 
$\sigma$. Namely, if $\kappa$ is a face of $\mathcal{N}(f)$, its {\bf dual cone}
\index{Dual cone!of a face}
$\boxed{\sigma(\kappa)}\subset\sigma$ is the closure of the set of weight vectors 
$w\in\sigma$ whose
restriction $w: \mathcal{N}(f) \to \mathbb{R}_{\geq 0}$ achieves its minimum exactly 
on the face $\kappa$. Then, as results from \cite[Chapter II.3.(A)]{O 97}, one has the following 
statement: 

\begin{lemma} \label{lem:Nfan}
      The set of cones $\sigma(\kappa)$, where $\kappa$ varies over the faces of the Newton  polyhedron 
       $\mathcal{N}(f)$ of a series $f\in \mathbb{K}[[\Gamma]]$, forms a fan with support $\sigma$.
\end{lemma}

\begin{definition}\label{def:Nfan}
      The fan $\boxed{\mathcal{F}(f)}$ with support $\sigma$
      described in Lemma \ref{lem:Nfan} is called the {\bf Newton fan} of the series $f$.
      \index{Newton!fan}  \index{Fan!Newton}
\end{definition}   

For instance, the subdivision of the non-negative quadrant shown in Figure \ref{fig:loctropcubic} 
is the Newton fan of the series $f(x,y) = y^2 - 2 x^3 + x^2 y \ \in \mathbb{C}[[x,y]]$, whose 
Newton polygon is shown in Figure \ref{fig:NPcubicex}. 

By an argument similar to those explained for plane curve singularities after 
Proposition \ref{prop:NPMinksum}, the Newton polyhedron of a defining series $f$ 
of the interior subtoric principal divisor
$Y$ is independent of the choice of this series. Therefore, one may speak simply of the  
{\bf Newton fan} $\boxed{\mathcal{F}(Y)}$ of $Y \hookrightarrow (\mathcal{X}_{\sigma},0)$.

The connection of the Newton fan $\mathcal{F}(Y)$ with the local tropicalization 
of $Y$ (see Theorem \ref{thm:coincidthm}), is given by the following 
result of \cite[Proposition~11.8]{PS 13}:

\begin{theorem}    \label{thm:divloctrop}
     Let $Y \hookrightarrow (\mathcal{X}_{\sigma},0)$ be an interior effective principal divisor. 
     Then, the local tropicalization $\mathrm{Trop}_{\mathrm{loc}} Y$ of $Y$
     coincides with the union of the cones of the Newton fan $\mathcal{F}(Y)$ of $Y$ 
     which are dual to the compact edges of $\mathcal{N}(Y)$.
     \index{Local tropicalization!of a principal divisor} 
\end{theorem}

One may prove Theorem \ref{thm:divloctrop} similarly to the proofs of Theorem \ref{thm:pqpair} 
and Corollary \ref{cor:arcweightplanecurve}, starting from characterization \eqref{eq:closarcwise} in   
Theorem \ref{thm:coincidthm} of the local tropicalization $\mathrm{Trop}_{\mathrm{loc}} Y$ of $Y$.

\begin{figure}[h!] 
\begin{center}
\begin{tikzpicture}[x=1.1cm,y=1.1cm, fill opacity=0.8] 

\begin{scope}[shift={(0,0)}]

\draw[->, thick] (0,0) -- (3,-1)  ;
\draw[->, thick] (0,0) -- (4,1)  ;
\draw[->, thick] (0,0) -- (0,3)  ;

 \fill[fill=blue!80]  (2,-0.66)--(3,0.75)-- (0, 2) --cycle;
   \fill[fill=yellow!50]  (0,3) -- (0, 2) -- (3,0.75) --(4,1) --cycle;
   \fill[fill=green!20]  (2,-0.66) -- (3,-1) --(4,1) -- (3,0.75)--cycle;
     \fill[fill=pink]  (0,3) -- (0, 2) --(2,-0.66) -- (3,-1) --cycle;

\node[draw,circle,inner sep=1.5pt,fill=black, color=black] at (2,-0.66) {};
\node[draw,circle,inner sep=1.5pt,fill=black,color=black] at (3,0.75) {};
\node[draw,circle,inner sep=1.5pt,fill=black,color=black] at (0,2) {}; 

\node [left] at (0,2) {$x_1^{\alpha}$};
\node [below] at (2,-0.66) {$x_2^{\beta}$};
\node [above] at (3, 0.75) {$x_3^{\gamma}$};

\node [below] at (1.5, -1.5) {$\mathcal{N}(Y) \subseteq \sigma_0^{\vee}$};

\end{scope}


\begin{scope}[shift={(6,0)}]

\draw[->, thick] (0,0) -- (3,-1)  ;
\draw[->, thick] (0,0) -- (4,1)  ;
\draw[->, thick] (0,0) -- (0,3)  ;

 \fill[fill=green!20]  (2.5, 3.12) -- (0,3) --(0,0) --cycle;
  \fill[fill=pink]  (2.5, 3.12) -- (4, 1) --(0,0) --cycle;
  \fill[fill=yellow!50]   (2.5, 3.12) -- (3, -1) --(0,0) --cycle;

\draw[-, blue, thick] (0,0) -- (2.5, 3.12)  ;

\node[draw,circle,inner sep=1.5pt,fill=black, color=black] at (2,2.5) {};

\node [right] at (2,2.5) {$(\beta \gamma, \gamma \alpha, \alpha \beta)$};
\node [below] at (1.5, -1.5) {$\mathrm{Trop}_{\mathrm{loc}} Y \subseteq \sigma_0$};

\end{scope}

 \end{tikzpicture}
\end{center}
\caption{The Newton polyhedron $\mathcal{N}(Y)$ of the Pham-Brieskorn singularity $Y \hookrightarrow (\mathbb{C}^3,0)$ defined by equation $x_{1}^{\alpha} + x_{2}^{\beta} +  x_{3}^{\gamma}=0$  
(on the left) and its local tropicalization (on the right).}  
   \label{fig:NPLTPB}
    \end{figure}
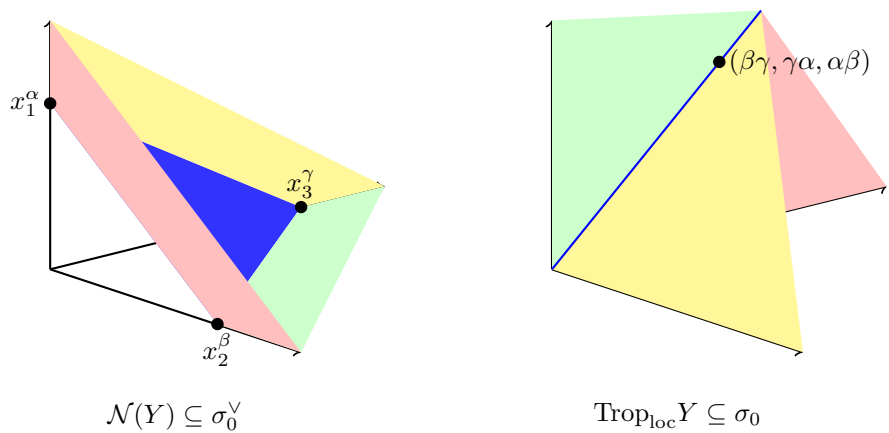

\begin{example}   \label{ex:loctropPhanBries}
   Consider a  {\bf Pham-Brieskorn surface singularity} \index{Singularity!Pham-Brieskorn}
   (a name which alludes mainly  to the papers \cite{P 65} 
   and \cite{B 68}, as explained in \cite[Pages 42--49]{B 00}). It is by definition a hypersurface 
   singularity in $\mathbb{C}^3$ defined by an equation of the form 
\begin{equation}   \label{eq:Brieskorn}
       x_{1}^{\alpha} + x_{2}^{\beta} +  x_{3}^{\gamma}=0,
\end{equation}
where $\alpha, \beta, \gamma$ are pairwise coprime positive integers. Its Newton polyhedron 
$\mathcal{N}(Y)$ and its local tropicalization $\mathrm{Trop}_{\mathrm{loc}} Y$ are represented 
in Figure \ref{fig:NPLTPB}. The weight cone $\sigma_0$ of $\mathbb{C}^3$ seen as an affine 
toric variety is the non-negative octant of the weight space $\mathbb{R}^3$. The compact faces 
of $\mathcal{N}(Y)$ are the triangle with vertices $(\alpha, 0, 0), (0, \beta, 0), (0, 0, \gamma)$, 
its edges and its vertices. By Theorem \ref{thm:divloctrop}, $\mathrm{Trop}_{\mathrm{loc}} Y$ is the 
 union of the three two-dimensional cones dual to the edges of the previous triangle. 
 Those cones are spanned by the vector 
 $(\beta \gamma, \gamma \alpha, \alpha \beta) \in \sigma_0$ orthogonal to the triangle and 
 by each of the three vectors of the canonical basis of the weight space $\mathbb{R}^3$. 
\end{example}

Note that Andr\'as N\'emethi and Baldur Sigur\dh sson described in \cite{NS 21} 
the local tropicalizations of effective but not necessarily principal divisors on toric germs.

\medskip
\section{The local tropicalization of splice type surface singularities}  \label{sec:splicetype}
\medskip

Theorem \ref{thm:divloctrop} shows that  the local tropicalization of an interior effective 
principal subtoric divisor is determined by the associated Newton polyhedron. By contrast, 
if an interior subtoric germ $Y \hookrightarrow (\mathcal{X}_{\sigma},0)$ has a non-principal 
defining ideal $I_Y$, then it is much more difficult to determine its local tropicalization. 
The reason is that  $\mathrm{Trop}_{\mathrm{loc}} Y$ is in general not the intersection 
of the local tropicalizations of the members of a generating family of $I_Y$, even if this 
generating family describes $Y$ as a complete intersection.  In this section we explain 
how we described in \cite{CPS 24}, in collaboration with Mar\'{\i}a Ang\'elica Cueto, 
the local tropicalizations of {\em splice type singularities}, which form an important 
class of isolated complete intersection surface singularities. 
\medskip

{\bf Splice type surface singularities}  \index{Singularity!splice type}
were introduced in \cite{NW 05bis} and \cite{NW 05} by Walter Neumann 
and Jonathan Wahl as examples  of complex isolated complete intersection 
surface singularities whose links are integral homology spheres. 
One may refer to Wahl's papers \cite{W 06} and \cite{W 22} for a description 
of their motivations for defining them. 

This class of singularities contains the class of Pham-Brieskorn surface  singularities  (see Example \ref{ex:loctropPhanBries}).  The following system defines also a splice type surface singularity in $\mathbb{C}^4$ (see \cite[Examples 2.27, 2.34, 2.40]{CPS 23}):
\begin{equation}\label{eq:spliceex}
\begin{cases}
    x_1^2 - x_2^3 + x_3 x_4 = 0, \\
    x_3^5 - x_4^2 + x_1 x_2^4 = 0.
  \end{cases}
\end{equation}

In general, a splice type surface singularity is determined up to a special kind of 
equisingular deformation by a {\bf splice diagram}. \index{Splice diagram}
By this, one means a pair $(\Gamma,w)$, where $\boxed{\Gamma}$ is a tree whose
vertices have valencies either $1$ ({\bf leaves}) or $\geq 3$ ({\bf nodes}) and $w$ is a function
that assigns to each pair $(u,e)$ of a node $u$ and adjacent edge $e$ a positive integer ({\bf weight})
$\boxed{w_{u,e}}$. Thus, each edge connecting nodes $u$ and $v$ carries two weights 
-- one near the node $u$ and one near $v$. In the context of integral homology sphere links, 
these data must satisfy three more properties, called {\em coprime, edge determinant} and 
   {\em semigroup conditions}:
\begin{itemize}
    \item The {\em coprime condition} states that 
        \emph{for each node, the weights of adjacent edges are $\geq 2$ and  pairwise coprime}. 
        
     \item For each node $u$ and vertex $p$ 
            (which may be either a node or a leaf), define $\boxed{\ell_{u,p}}$ to be the
            product of all the weights adjacent to, but not lying on, the shortest path from $u$ to $p$.  
            In particular, when $p =u$ one gets the product $\boxed{d_u}$ of weights adjacent to $u$. 
           The  {\em edge determinant condition} states that {\em for every pair of adjacent nodes 
           $u$ and $v$, one has the inequality $d_u d_v > \ell_{u,v}^2$}.
           
     \item Denote by $1,\ldots, n$ the leaves of $\Gamma$. For each node $u$, 
         consider the {\bf weight vector} 
        \[   \boxed{w_u}=(\ell_{u,1},\dots,  \ell_{u,n})\in \mathbb{Z}^n\subset \mathbb{R}^n. \]
        If $e$ is an edge adjacent to $u$, then it makes sense to speak about the leaves of $\Gamma$ 
        seen from $u$ in the direction $e$. If $i_1,\ldots, i_s$ are these leaves, then the 
        {\em semigroup condition} says that \emph{the number $d_u$ must be contained in 
       the sub-semigroup of $(\mathbb{Z}_{\geq 0}, +)$ generated by 
       $\ell_{u,i_1},\ldots,\ell_{u,i_s}$, for all pairs $(u,e)$ of a node $u$ and an adjacent edge $e$}.
 \end{itemize}
Examples of splice diagrams satisfying the three conditions are shown in Figure \ref{fig:splicediags}.

\begin{figure}[h!] 
\begin{center}
\begin{tikzpicture}[x=0.7cm,y=0.7cm] 

\begin{scope}[shift={(0,0)}]

\draw[-] (-2,-2) -- (0,0) -- (-2,2)  ;
\draw[-] (0,0) -- (3,0)  ;

\node[draw,circle,inner sep=1.5pt,fill=black, color=black] at (0,0) {};
\node[draw,circle,inner sep=1.5pt,fill=black,color=black] at (-2,-2) {};
\node[draw,circle,inner sep=1.5pt,fill=black,color=black] at (-2,2) {};
\node[draw,circle,inner sep=1.5pt,fill=black,color=black] at (3,0) {};

\node [above] at (-0.5,0.5) {$\alpha$};
\node [below] at (-0.5,-0.5) {$\beta$};
\node [above] at (0.7, 0) {$\gamma$};
\node [above] at (-2,2) {$x_1$};
\node [below] at (-2,-2) {$x_2$};
\node [above] at (3, 0) {$x_3$};

\end{scope}


\begin{scope}[shift={(8,0)}]

\draw[-] (-2,-2) -- (0,0) -- (-2,2)  ;
\draw[-] (0,0) -- (3,0)  ;
\draw[-] (5,-2) -- (3,0) -- (5,2)  ;

\node[draw,circle,inner sep=1.5pt,fill=black, color=black] at (0,0) {};
\node[draw,circle,inner sep=1.5pt,fill=black,color=black] at (-2,-2) {};
\node[draw,circle,inner sep=1.5pt,fill=black,color=black] at (-2,2) {};
\node[draw,circle,inner sep=1.5pt,fill=black,color=black] at (3,0) {};
\node[draw,circle,inner sep=1.5pt,fill=black,color=black] at (5,-2) {};
\node[draw,circle,inner sep=1.5pt,fill=black,color=black] at (5,2) {};

\node [above] at (-0.5,0.5) {$2$};
\node [below] at (-0.5,-0.5) {$3$};
\node [above] at (0.7, 0) {$7$};
\node [above] at (2.3,0) {$11$};
\node [above] at (3.5,0.5) {$2$};
\node [below] at (3.5, -0.5) {$5$};
\node [above] at (-2,2) {$x_1$};
\node [below] at (-2,-2) {$x_2$};
\node [below] at (5, -2) {$x_3$};
\node [above] at (5, 2) {$x_4$};

\end{scope}

 \end{tikzpicture}
\end{center}
\caption{The splice diagrams corresponding to the singularities defined by equation 
       \eqref{eq:Brieskorn} (on the left) and by the system \eqref{eq:spliceex} (on the right).}  
   \label{fig:splicediags}
    \end{figure}
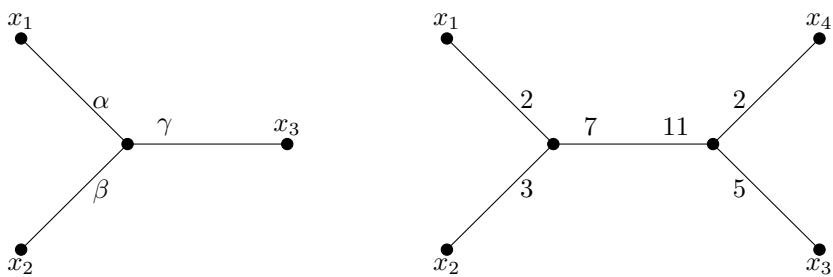

To get from such data a singularity with integral homology sphere link, one should apply the
following procedure developed by Neumann and Wahl. First of all, associate a variable $x_i$ 
to every leaf $i$. Then, let $u$ be a node, $e$ an adjacent edge, and
$i_1,\ldots,i_s$ the leaves seen from $u$ in the direction of $e$. Choose a representation
      \[  d_u=\sum_{k=1}^{s} n_k\ell_{e,i_k},   \]
such that $n_k \in \mathbb{Z}_{\geq 0}$ for every $k \in \{1, \dots, s\}$. Such a representation
exists, by the semigroup condition. Define the {\bf admissible monomial} $\chi_{u,e}$ as
    \[   \boxed{\chi_{u,e}} : =x_{i_1}^{n_1} \cdots x_{i_s}^{n_s}.  \]
If the node $u$ has valency $r_u\geq 3$ and $e_1,\ldots e_{r_u}$ are the adjacent edges, it is more
convenient to denote the corresponding admissible monomials by $\chi_{u,1},\ldots,\chi_{u,r_u}$.
For each node $u$, choose a complex matrix $\boxed{(a_{u,i,j})_{i,j}}$ 
of size $(r_u-2)\times r_u$,  such that all its maximal minors are non-zero. 
The {\bf splice type system} determined by the splice diagram
$(\Gamma,w)$ and a choice of admissible monomials $\chi_{u,j}$ and coefficients $a_{u,i,j}$ 
is the system of equations
    \begin{equation}\label{eq:splice}
        \sum_{j=1}^{r_u} a_{u,i,j} \chi_{u,j}=0, \:  \mbox{ for all }  i \in \{ 1,\dots,r_u-2 \},
    \end{equation}
where $u$ varies among all the nodes of $\Gamma$. It is an easy combinatorial exercise to check that  
the system~\eqref{eq:splice} has $n-2$ equations. Finally,
by definition, a {\bf splice type singularity} is the germ at the origin of
the variety in $\mathbb{C}^n$ defined by the system \eqref{eq:splice} or by any system obtained 
by adding to \eqref{eq:splice} terms of heigher $w_u$-weight to the equations corresponding to the 
node $u$.
The interested reader can check that splice diagrams shown
in Figure \ref{fig:splicediags} indeed correspond to the singularities defined by equation \eqref{eq:Brieskorn} and by the system \eqref{eq:spliceex}.

A system of equations corresponding to a big splice diagram can be rather complicated, 
and {\em it is far from obvious that the singularity defined by such a system is an irreducible 
isolated complete intersection of dimension $2$}. Indeed, the proof of this fact spans several pages
in \cite{NW 05bis} and \cite{NW 05}, where the authors proceed by direct calculations with partial
resolutions of the germ. In turn, the theory of local tropicalization allows for a more transparent and
conceptual proof of these properties. We developed this strategy in full detail in \cite{CPS 24}. 
Here we present a brief summary of our method.
\medskip

Our main point in  \cite{CPS 24} was to give a previously unknown tropical interpretation of 
the splice diagram. We associated above a weight vector $w_u$ to each node $u$ of a splice
diagram $\Gamma$. If we additionally associate the $i$-th vector $e_i$ of the canonical basis of $\mathbb{R}^n$ to the $i$th leaf of $\Gamma$, we proved in \cite[Theorem 5.11]{CPS 24} that once 
conveniently normalized, \emph{those weight vectors determine a piecewise linear 
embedding of the splice diagram $\Gamma$  into the simplex 
$\Delta_{n-1}$ whose vertices are the points $e_i$}.

Further, by analyzing by purely combinatorial means  (no calculations 
with resolutions being needed) the algorithm by which Neumann and Wahl 
constructed a system of equations from a given splice
diagram $\Gamma$, we showed in \cite[Theorem 1.2]{CPS 24} that:

\begin{theorem}
    The local   tropicalization of a splice type surface singularity $Y \hookrightarrow (\mathbb{C}^n, 0)$ 
    corresponding to the splice diagram $\Gamma$  is the non-negative cone in 
    $\mathbb{R}^n$ over the tree $\Gamma$ embedded in $\Delta_{n-1}$.  
 \end{theorem}

For instance, note that by intersecting the local tropicalization of the Pham-Brieskorn 
singularity represented on the right side of Figure \ref{fig:NPLTPB} with the triangle $\Delta_2$, 
one gets a tree isomorphic to the corresponding splice diagram, which is represented 
on the left side of Figure \ref{fig:splicediags}.

Since $\Gamma\subset \Delta_{n-1}$ is a $1$-dimensional simplicial complex, 
the local tropicalization $\mathrm{Trop}_{\mathrm{loc}} Y$ (endowed with appropriate fan structure) 
has maximal cones of dimension $2$. Note that we started with integral weight vectors 
corresponding to the vertices of $\Gamma$, thus $\mathrm{Trop}_{\mathrm{loc}} Y$
is indeed the support of a {\em rational} polyhedral fan, as stated in Theorem \ref{thm:structloctrop}.
From the same theorem, we got $\dim Y = 2$. The splice system corresponding to $\Gamma$ 
has the correct number of equations (namely $n-2$), so it follows at once that $Y$ 
is indeed a complete intersection and that all its irreducible components have dimension $2$.

Recall that throughout this paper we assumed for simplicity that all the subtoric germs were 
{\em interior} in the sense of Definition \ref{def:subtorgerm}. However, the theory of local 
tropicalization may be extended to subtoric germs which are not interior by adding 
\emph{strata at infinity} to the local tropicalization 
(see Subsection \ref{ssec:extloctrop} for an introduction and \cite[Sections 4 and 6]{PS 13} for details). 
In \cite{CPS 24} we performed the corresponding analysis and we showed that
if $Y$ is a splice type surface singularity, then its local tropicalization 
$\mathrm{Trop}_{\mathrm{loc}} Y$ has no strata at infinity
except the closures of its finite strata, and thus $Y$ has no irreducible components 
in the toric boundary $\partial \mathbb{C}^n$.

As a byproduct of our techniques, we got a relatively easy proof of the 
following new result (see  \cite[Theorem 1.1]{CPS 24}): 

\begin{theorem}
   Any splice type surface singularity is a Newton nondegenerate complete intersection.
\end{theorem} 

Combining this fact with Theorem \ref{thm:coincidthm} and the notion of
{\em embedded toric resolution} (see, e.g., \cite{K 84}), we proved that a resolution 
of singularities of $Y$ could be obtained by a toric morphism corresponding to a convenient fan 
supported by $\mathrm{Trop}_{\mathrm{loc}} Y$. But since the entire local tropicalization (with the exception
of the coordinate rays $\mathbb{R}_{\geq 0}\ e_i$) is contained in the positive orthant $\mathbb{R}_{>0}^{n}$
of $\mathbb{R}^n$, this toric morphism does not blow up anything except the origin in $\mathbb{C}^n$. This implies that
{\em the splice type singularity $Y$ is isolated}.

Finally, {\em $Y$ must be irreducible} because if a complete intersection singularity 
of dimension $\geq 2$ has two or more irreducible components, then it cannot be isolated 
(this is a special case of a theorem of Hartshorne, as explained in the proof of 
\cite[Corollary 7.13]{PS 13}).

\medskip

As a generalization of Goldin and Teissier's 
\cite[Corollary of Theorem 6.1]{GT 00}, 
which concerned only {\em branches}, 
de Felipe, Gonz\'alez P\'erez and Mourtada 
showed in \cite{FGM 23} that any plane curve singularity may be reembedded 
explicitly in a higher dimensional space $\mathbb{C}^n$ as a Newton non-degenerate 
germ and they expressed the local tropicalization of the 
reembedded germ in terms of invariants of the starting plane curve singularity. 
In \cite[Corollary 7.14]{CPS 24} we obtained a second proof of part of their 
reembedding theorem as a particular case of our general results about splice type singularities.

\medskip
\section{Variants and extensions of local tropicalization}  \label{sect:variants}
\medskip

In this final section we describe briefly a natural extension of the notion of local tropicalization to the
case of a general subtoric germ, which can possess components in the toric boundary 
(see Subsection \ref{ssec:extloctrop}), as well as
other generalizations by several authors (see Subsections \ref{ssec:Ulogtrop} and \ref{ssec:othergen}).

\subsection{Extended local tropicalization}    \label{ssec:extloctrop}  
$\ $ \medskip

\index{Local tropicalization!extended} 
Till now we have always assumed that the subtoric germ 
$Y \hookrightarrow (\mathcal{X} _{\sigma}, 0)$ was \emph{interior} in the sense of 
Definition \ref{def:subtorgerm}, that is, without irreducible components contained 
in the toric boundary $\partial \mathcal{X} _{\sigma}$. However, it is sometimes
necessary to relax this condition. It is the fourth definition of the local tropicalization, based on valuative
weight vectors (see point {\em \ref{eq:closvalweight}.} of Theorem \ref{thm:coincidthm}), 
that most naturally can be generalized.

To illustrate the difficulty that arises when the subtoric germ is not interior, let us start with an example. 
Let  again $\mathbb{K}$ be an algebraically closed field and $\mathcal{X}_\sigma=\mathrm{Spec} (\mathbb{K}[x_1,\ldots,x_n])$ be the affine 
 space considered as a toric variety. Assume that the subtoric germ $Y \hookrightarrow
 (\mathcal{X}_\sigma,0)$ is contained in the coordinate hyperplane defined by $x_1=0$, but in no other
coordinate hyperplane. This means that $x_1\in I_Y$, but $x_i\notin I_Y$ 
for any $i\ne 1$, where 
$I_Y  \subset \mathbb{K}[[x_1,\ldots,x_n]]$ is the defining ideal of $Y$ in 
$ \mathcal{O}_{\mathcal{X}_\sigma,0}$. 
Now, if we try to compute 
the local tropicalization of $Y$ according, say, to point {\em \ref{eq:closvalweight}.} of 
Theorem \ref{thm:coincidthm},
we see that since $r_Y(x_1)=0 \in \mathcal{O}_Y$, we must have $\nu\circ r_Y(x_1)=+\infty$ 
for {\em each} semivaluation $\nu$ of $Y$, where the morphism of monoids 
    \[ r_Y : (\Gamma, +) \to  (\mathcal{O}_Y, \cdot) \]
was introduced in Definition \ref{def:valweight}. Thus there are no interior semivaluations and the local
tropicalization is empty. To get something nontrivial, we have to allow the semivaluation $\nu$ 
to take infinite
values on $\Gamma$, but then none of the corresponding valuative weight vectors of $Y$ can be 
interpreted as an element of the real vector space $N_{\mathbb{R}}$.

To see how to deal with this issue, let us denote by $\Gamma'$ the submonoid of 
$\Gamma = \mathbb{Z}_{\geq 0}^{n}$ 
generated by the exponent vectors of the monomials $x_2,\ldots,x_n$. Observe that despite the fact that any semivaluation $\nu$ of $Y$ takes value
$+\infty$ on $x_1$, such a semivaluation can be finite on $r_Y(\Gamma')$.
If this is the case, then $\nu$ defines a linear form on the space $M'_{\mathbb{R}}$, where $M'$ is
the sublattice of $M$ generated by $\Gamma'$. 

Recall the simple fact from linear algebra that
if $U$ is a linear subspace of a finite dimensional vector space $V$, then a linear form on $U$ is 
the same thing as an element of the quotient space $V^\vee/U^\perp$, i.e., there exists a canonical
isomorphism
     \[ U^\vee \simeq V^\vee/U^\perp, \]
where $V^\vee$ is the dual space of $V$ and $U^\perp \subseteq V^{\vee}$ is the orthogonal of $U$. It follows that we can interpret the general valuative weight vectors of $Y$ as elements of appropriate
quotients of the space $N_{\mathbb{R}}$. Following an approach initiated by Ash, Mumford, Rapoport and Tai in \cite[Section I.1]{AMRT 75} and developed by Kajiwara \cite{K 08} and Payne \cite{P 09}, we regard those quotients as \emph{strata at infinity} that can be
adjoined to the space $N_{\mathbb{R}}$.

The formal definition proceeds as follows. Let as before $\sigma$ be a strictly convex 
rational polyhedral cone in $N_{\mathbb{R}}$. If $\tau$ is a face of $\sigma$, denote $\boxed{N_\tau} := N\cap \mathbb{R} \tau$, where $\mathbb{R} \tau$ is the real vector subspace of $N_{\mathbb{R}}$ spanned by $\tau$. Since the abelian group $N_\tau$ is a {\em saturated} sublattice of $N$, 
then the quotient $N/N_\tau$ is also a lattice. Consider the disjoint union: 
   \begin{equation}\label{eq:linearvar}
        \boxed{L(\sigma,N)} :=\bigsqcup_{\tau\leq\sigma} (N/N_\tau)_{\mathbb{R}},
   \end{equation}
 where $\tau$ varies in the set of faces of the polyhedral cone $\sigma$. 
 
It is a formal exercise to see that $L(\sigma,N)$ can be identified with the set of all monoid homomorphisms from $(\Gamma := \sigma^\vee \cap M, +)$ to 
$(\boxed{\overline{\mathbb{R}}} :=\mathbb{R}\cup\{+\infty\}, +)$. 
Moreover, using the standard
topology on $\overline{\mathbb{R}}$, for which a basis of neighborhoods of $+\infty$ is the set of all intervals of the form $(a,+\infty]$, we can endow $L(\sigma,N)$ with the topology of pointwise convergence. Then, \eqref{eq:linearvar} represents $L(\sigma,N)$ as a
disjoint union of locally closed subsets, which we call its {\bf strata}.

Each stratum $(N/N_\tau)_{\mathbb{R}}$ carries naturally an integral lattice $N/N_\tau$.  
The cone $\sigma \subset N_{\mathbb{R}}$ has images in all the strata $(N/N_\tau)_{\mathbb{R}}$, 
by the canonical projections $N_{\mathbb{R}} \to (N/N_\tau)_{\mathbb{R}}$. We say that the disjoint 
union of $\sigma$ with all those images is the {\bf extended cone} \index{Weight cone!extended}
$\boxed{\overline{\sigma}}$.  
It is a closed topological subspace of $L(\sigma,N)$ which can be alternatively defined as the closure
of the cone $\sigma$ in $L(\sigma,N)$. The topological spaces $L(\sigma,N)$ had appeared several
times in the literature, but since there seemed to be no standard name for them, we called them 
{\bf linear varieties} \index{Linear variety} 
in \cite[Section 4]{PS 13}. We chose that name by analogy with that of 
{\em toric variety}. Indeed, similarly to the fact that a toric variety is endowed with an action of an algebraic torus embedded in it as a dense open set, a linear variety is endowed with an action of a linear space which is also embedded in it as a dense open set (see \cite[Remark 4.4]{PS 13}).

\begin{definition}   \label{def:extloctrop}
     \index{Local tropicalization!extended} 
    Let $Y \hookrightarrow (\mathcal{X},0)$ be a subtoric germ, not assumed to be interior. 
    Then, its {\bf extended local tropicalization} is defined as the closure in the extended cone  
    $\overline{\sigma}$ of all valuative weight vectors of $Y$ of the form $\nu\circ r_Y$ 
    (see Definition \ref{def:valweight}), where $\nu$ ranges over all (not necessarily interior) 
    semivaluations of $Y$.
\end{definition}

The notion of rational polyhedral fan has its extended version, and an analog of the structure
Theorem \ref{thm:structloctrop} holds for extended local tropicalization. In fact, the main finiteness theorem
is formulated and proven in \cite[Theorem~11.9]{PS 13} for extended local tropicalization 
(which we called there {\em non-negative local tropicalization}). 

In the case when $Y$ is not contained in the toric boundary of $\mathcal{X}$, 
Definition \ref{def:extloctrop}
generalizes Definition \ref{def:loctrop} of local tropicalization, 
in the sense that $\mathrm{Trop}_{\mathrm{loc}} Y$ is obviously
contained in the extended local tropicalization. Then, the question arises whether the extended and
non-extended tropicalizations are essentially different from each other. The answer is {\em no}, and this
again follows formally from \cite[Theorem~11.9]{PS 13}:

\begin{proposition}
     If $Y\hookrightarrow (\mathcal{X}_{\sigma},0)$ is an interior subtoric germ, 
     then the extended local tropicalization 
     of $Y$ is the closure  in $L(\sigma,N)$ of the local tropicalization $\mathrm{Trop}_{\mathrm{loc}} Y$ 
     in the sense of Definition \ref{def:loctrop}.
\end{proposition}

\begin{remark}
      We have to warn the reader wishing to continue studying the local tropicalization from \cite{PS 13}
      about a terminological difference with the present paper. 
      In \cite[Definitions 6.6 and 6.7]{PS 13}, we distinguished between
      the {\em positive} and {\em nonnegative local tropicalization},
        \index{Local tropicalization!positive} \index{Local tropicalization!non-negative} 
       whereas here we speak simply
      about {\em local tropicalization}. In fact, the main Definitions \ref{def:loctrop} and \ref{def:extloctrop}
       of this paper correspond to the nonnegative version of \cite{PS 13}. The difference 
       with the positive
      local tropicalization can be explained as follows. Let $Y$ be a subgerm of an affine toric germ
     $(\mathcal{X}_\sigma,0)$. Then we get the positive local tropicalization if in 
     Definition \ref{def:extloctrop}
    we consider only the semivaluations whose \emph{center} is exactly the unique closed orbit $0$
    of $\mathcal{X}_\sigma$, and the closure operation is performed in the \emph{interior} of the extended
    cone $\overline{\sigma}$. If we allow the semivaluations whose center only \emph{contains} the unique 
    closed orbit (as is implicitly done in Definition~\ref{def:extloctrop}, see also the very 
    Definition \ref{def:semival} of semivaluation used in this paper), then we get the nonnegative local
    tropicalization. The two local tropicalizations determine each other, see 
    \cite[Theorems~11.2, 11.9]{PS 13}. In this introductory text we felt no need to overload 
    the reader's attention with such technical details.
\end{remark}

To give the reader a better feeling of what is the extended tropicalization, let us consider one more
example. In all our definitions, it creates no problem to take the germ $Y$ to be 
{\em equal} with the germ 
$(\mathcal{X}_{\sigma},0)$ of affine toric variety itself. {\em Let us see what are the local
tropicalization of $(\mathcal{X}_{\sigma},0)$ and its extended version.}  
Assume that the base field $\mathbb{K}$ is the field $\mathbb{C}$ of complex numbers. First recall 
from Section \ref{sect:notations} that the set of $\mathbb{C}$-valued points 
$\mathcal{X}_{\sigma}(\mathbb{C})$ of the toric variety $\mathcal{X}_{\sigma}$
corresponding to a cone $\sigma\subset N_{\mathbb{R}}$ can be identified with the set of 
\emph{monoid homomorphisms} $\sigma^\vee\cap M\to \mathbb{C}$, where $\mathbb{C}$ 
is considered as a monoid with respect to
multiplication. If we replace $(\mathbb{C}, \cdot)$ by $(\mathbb{R}, \cdot)$, we get in a similar way 
the set of real points $\mathcal{X}_{\sigma}(\mathbb{R})$ of $\mathcal{X}_{\sigma}$,
and, if we replace $(\mathbb{C}, \cdot)$ by the monoid $(\mathbb{R}_{\geq 0}, \cdot)$ 
of nonnegative real numbers, 
we get its set $\mathcal{X}_{\sigma}(\mathbb{R}_{\geq 0})$ of nonnegative
real points. The group $(\mathbb{R}_{>0}, \cdot)$ is isomorphic via the continuous 
isomorphism $\boxed{\log}$ to the group $(\mathbb{R}, +)$, 
and if we extend $\log$ by setting formally $\log 0 :=-\infty$, we get an isomorphism
    \[  -\log\colon (\mathbb{R}_{\geq 0}, \cdot) \to (\overline{\mathbb{R}}, +)   \]
of monoids. Note that the subset $[0,+\infty]$ of $\overline{\mathbb{R}}$ corresponds under this isomorphism to the segment $[0,1]$ of $\mathbb{R}_{\geq 0}$. 

By composing any monoid morphism 
 $ \rho \colon (\sigma^{\vee}  \cap M, +)  \to (\mathbb{R}_{\geq 0}, \cdot)$ with $-\log$, 
 we get a monoid morphism 
    \[(-\log)\circ  \rho \colon (\sigma^\vee\cap M, +)   \to (\overline{\mathbb{R}}, +).  \]
    If such a morphism does not take the value $+\infty$, then it is simply an element 
  of $\sigma\subset N_{\mathbb{R}}$. Otherwise, it is an element of 
  $\overline{\sigma}\subset L(\sigma,N)$ belonging to a stratum at infinity, that is, 
  different from $N_{\mathbb{R}}$. It is easy to see that any such element can be realized 
  by a semivaluation on the formal local ring $\mathbb{C}[[\Gamma]]$ of $(\mathcal{X}_{\sigma}, 0)$. 
   For example, if $w\in \sigma^\circ$, we can define the associated semivaluation 
   (which is in fact a valuation) by:
     \[   \nu_w(f)  :=\min_{m\in \  \mathrm{supp}\  f} w\cdot m,  \ \ 
           \mbox{ for every } f\in \mathbb{C}[[\Gamma]]. \]

We conclude that: 

\begin{proposition}   \label{prop:loctroptoric}
     One has $\mathrm{Trop}_{\mathrm{loc}} (\mathcal{X}_{\sigma},0) =\sigma$, 
     whereas the extended 
     local tropicalization of $(\mathcal{X}_{\sigma}, 0) \hookrightarrow (\mathcal{X}_{\sigma}, 0)$ is 
     the extended cone $\overline{\sigma}$.   \index{Local tropicalization!of a toric germ} 
\end{proposition}

\begin{figure}[h!] 
\begin{center}
\begin{tikzpicture}[x=0.9cm,y=0.9cm, fill opacity=0.8] 

\begin{scope}[shift={(0,0)}]

 \fill[fill=yellow!50]  (0,0) -- (0, 3) -- (3.5 , 3) --(4,1) --cycle;

\draw[->, thick] (0,0) -- (4,1)  ;
\draw[->, thick] (0,0) -- (0,3)  ;

\draw[-, very thick] (0,3) -- (3.5,3)  ;
\draw[-, very thick] (4,1) -- (3.5,3)  ;
 
\node[draw,circle,inner sep=1.5pt,fill=orange, color=orange] at (3.5,3) {};

\node [left] at (0,0) {$0$};
\node [below] at (2,0.5) {$\lambda_1$};
\node [left] at (0, 1.5) {$\lambda_2$};
\node [left] at (3.8, 1.5) {$\sigma_{\lambda_1} \subset(N/N_{\lambda_1})_{\mathbb{R}}$};
\node [above] at (1.8, 3) {$\sigma_{\lambda_2} \subset(N/N_{\lambda_2})_{\mathbb{R}}$};

\end{scope}


\begin{scope}[shift={(8,0)}]

\draw[dotted, very thick] (4,0) -- (4,4) -- (0,4)  ;
\draw[dotted, very thick] (5,6) -- (4,4)  ;
 \fill[fill=yellow!50]   (0, 0) -- (4, 0) --(5,1) --(5,6) -- (1,5) -- (0,4) -- cycle;
  \fill[fill=blue!20]  (0, 0) -- (4,0) --(5,1) -- (2,1) --cycle;
  \fill[fill=pink]  (1, 2) -- (2, 2) --(5,6) -- (1,5) --cycle;

\draw[->, thick] (0,0) -- (4,0)  ;
\draw[->, thick] (0,0) -- (2,1)  ;
\draw[->, thick] (0,0) -- (0,4)  ;

\draw[-, very thick] (4,0) -- (5,1) -- (2,1) -- (2,2) -- (1,2) -- (1,5) -- (0,4)  ;
\draw[-, very thick] (5,6) -- (1,5)  ;
\draw[-, very thick] (5,6) -- (2,2)  ;
\draw[-, very thick, blue] (5,6) -- (5,1)  ;
 
\node[draw,circle,inner sep=1.5pt,fill=orange, color=orange] at (5,6) {};

\draw[->, very thick, red] (0,0) -- (1,2)  ;

\node [left] at (0,0) {$0$};
\node [left] at (1,1.2) {\red{$\lambda$}};
\node [below] at (2, 0.6) {$\tau$};
\node [left] at (5, 2) {\blue{$\sigma_{\tau} \subset(N/N_{\tau})_{\mathbb{R}}$}};
\node [above] at (2.5, 4.5) {$\sigma_{\lambda} \subset(N/N_{\lambda})_{\mathbb{R}}$};

\end{scope}

 \end{tikzpicture}
\end{center}
\caption{The extended local tropicalizations $\overline{\sigma}$ of a germ of affine toric variety 
         $(\mathcal{X}_{\sigma}, 0)$ of dimension $2$ on the left and dimension $3$ on the right 
        (see Example \ref{ex:extloctrop}).}  
   \label{fig:extloctropafftoric}
    \end{figure}
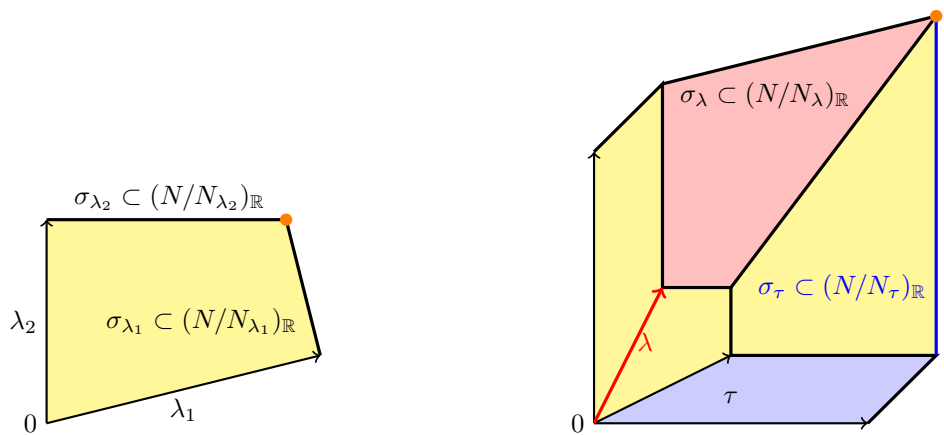

    \begin{example}   \label{ex:extloctrop}
         In Figure \ref{fig:extloctropafftoric} are shown the structures of the extended local tropicalizations of 
         formal germs $(\mathcal{X}_{\sigma}, 0)$, for a $2$-dimensional cone $\sigma$ on the left 
         and a three-dimensional one with four rays on the right (see also \cite[Figures 3 and 4]{PS 13}). 
         If $\zeta$ is a face of the cone $\sigma$, we denote by $\sigma_{\zeta}$ the image of 
         $\sigma$ by the quotient linear map $N_{\mathbb{R}} \to (N/N_{\zeta})_{\mathbb{R}}$. 
         It is the intersection of 
         the extended cone $\overline{\sigma}$ with the stratum $ (N/N_{\zeta})_{\mathbb{R}}$. 
         We denote by 
         $\lambda_1, \lambda_2$ the edges of the left-hand example and by $\lambda, \tau$ an edge 
         and a $2$-dimensional face respectively
         of the right-hand $3$-dimensional example. In both drawings 
         we have indicated by a filled disk the $0$-dimensional stratum $(N/N_{\sigma})_{\mathbb{R}}$ at 
         infinity. 
    \end{example}

In \cite[Section 12]{PS 13} we have compared local tropicalizations with the global 
tropicalization of subvarieties of algebraic tori. We showed in particular (see \cite[Theorem 12.11]{PS 13}) that if $Y \hookrightarrow \mathcal{X}_{\sigma}$ 
is an algebraic subvariety of the affine toric variety $\mathcal{X}_{\sigma}$ with dense torus $T_N$,  
whose germ $Y$ is an interior toric subgerm, then the local 
tropicalization of $Y$ is equal to the intersection of the cone $\sigma$ with the global tropicalization of 
$Y \cap T_N \hookrightarrow T_N$. 

Local tropicalization has also {\em functorial properties}, 
described in \cite[Proposition 6.12]{PS 13}. This allowed us to generalize in
\cite[Section 13]{PS 13} the definition of local tropicalization to germs of subvarieties of toric 
(not necessarily affine) varieties, as well as to subvarieties of toroidal embeddings in the sense of
\cite{KKMS 73}. However, a much stronger generalization is possible, as we explain 
in Subsection \ref{ssec:Ulogtrop}.

\subsection{Ulirsch's tropicalization of logarithmic schemes}  \label{ssec:Ulogtrop} 
$\ $ \medskip

\index{Local tropicalization!of logarithmic schemes} 
The most basic piece of data which gives rise to a local tropicalization is a morphism of monoids 
\begin{equation}\label{eq:localmonmorphism}
      \gamma\colon (\Gamma, +) \to (\mathcal{O}, \cdot)
\end{equation}
where $\Gamma :=\sigma\cap M^\vee$ as before and $\mathcal{O}$ is a local ring, 
such that $\gamma$ sends all the non-units of $\Gamma$ to the maximal ideal of $\mathcal{O}$. 
If, from a general point of view, the choice of monoids $\Gamma$ of that particular
form seems not motivated enough, note that they can be characterized by abstract properties
as the monoids which are simultaneously {\em commutative, finitely generated, cancellative, torsion free}, and {\em saturated}. The monoids satisfying the first three properties are also called \emph{fine}. Morphisms of the form \eqref{eq:localmonmorphism} were the basic
building blocks used in \cite{PS 13} for the development of various generalizations of the local
tropicalizations, trying to stick to elementary methods. The theory of {\em logarithmic structures}, initiated by Fontaine and Illusie and first described by Kato in his 1989 paper \cite{K 89}, incorporates monoid morphisms of this sort in a very general framework. For details on 
this framework, one may consult Ogus' textbook \cite{O 18}. 

A {\bf logarithmic structure}   \index{Logarithmic structure}
on a ringed space $(X, \mathcal{O}_X)$ 
is a sheaf $\mathcal{M}_X$ of monoids together
with a morphism $\alpha\colon \mathcal{M}_X \to \mathcal{O}_X$ of sheaves of monoids, 
such that $\alpha$ induces an isomorphism $\alpha^{-1}(\mathcal{O}_X^*)\simeq \mathcal{O}_X^*$.

An important class of logarithmic structures  
comes from pairs $(X,D)$, where $X$ is a normal scheme and $D$ is a hypersurface 
of it. The associated {\bf divisorial logarithmic structure}   \index{Logarithmic structure!divisorial}
on $X$ is defined by
      \[ \mathcal{M}_X(U)   :=\{f\in \mathcal{O}_X(U) \,|\, f|_{U-D}\in \mathcal{O}_{X}^{*}(U)\},  \]
 for every open subset $U$ of $X$. That is, $\mathcal{M}_X$ is the sheaf of regular functions on $X$ which 
 do not vanish outside $D$.  Toric varieties and toroidal embeddings possess therefore canonical logarithmic structures, which are the divisorial logarithmic structures induced by their boundaries. 
Moreover, to each sufficiently good logarithmic scheme (namely a fine and saturated one), 
it is possible to associate a \emph{conical complex}, 
both in an extended and non-extended variants. Informally,
a conical complex is like a fan, but its maximal cones do not have to be contained in a fixed vector space. In \cite{U 17}, Martin Ulirsch developed a theory of tropicalization of logarithmic schemes
and their subvarieties. In the special case where $X$ is an affine toric variety, 
it recovers the local tropicalization described at the end of Subsection \ref{ssec:extloctrop}. 
In \cite{U 15}, Ulirsch proved an analog of (1) of 
Theorem \ref{thm:coincidthm} for the case of so called {\em log-regular varieties}.

We would like to mention one more aspect of the theory of logarithmic structures which is of particular
interest to singularity theory. Namely, this theory
provides also a general framework for the operation of \emph{rounding},
or \emph{real oriented blowing up}, of algebraic varieties and analytic spaces 
(see \cite{P 24}). This gives a new way
of approaching some problems connected to Milnor fibrations, as explained in 
\cite[Section~3.1]{CPS 23}.

\subsection{Other generalizations of local tropicalization}    \label{ssec:othergen}    
$\ $ \medskip

In \cite[Section 2]{E 21}, Alexander Esterov defined a notion of {\em local tropical fan}  
\index{Fan!local tropical} as an independent object, i.e., not necessarily coming from a tropicalization. 
Then,  in \cite[Definition 3.4]{E 21} he defined the local tropicalization of an analytic germ 
$Y \hookrightarrow \mathcal{X}_\sigma$ of pure dimension $k$
essentially as in point {\em \ref{eq:compactcone}.} of  Theorem \ref{thm:coincidthm},  
but endowing it with an additional structure, namely a {\em weight function} 
that assigns a positive integral weight  
to each maximal cone of a particular kind of fan $\mathcal{F}$ whose support is the 
local tropicalization. 
If $Y$ is of pure dimension $k$, then the weight of a $k$-dimensional cone $\tau$ equals
the intersection number  of $Y$ with the orbit $O_{\tau}$ corresponding to $\tau$ 
(see also \cite[Definition 3.20]{CPS 24}). 
Moreover, in order to develop an intersection theory of local tropical fans and of local tropical
characteristic classes he showed that one can associate to a $k$-dimensional subtoric germ not only the
tropicalization above, which is a $k$-dimensional local tropical fan, 
but a sequence of tropical fans of dimension 
$0,1,\ldots,k$, where the support of the $k$-dimensional fan is the local tropicalization, 
and which reflects more information about the germ. Esterov applied this extension 
of the theory of local tropicalization to the study of the 
monodromy of complex analytic functions.

We would like also to mention two recent developments in the theory of tropicalization. 
They are however not so much concerned with the local tropicalization as with the general 
question of determining which are the most general objects which may be tropicalized.

In \cite{GG 16} and \cite{GG 22}, Jeffrey and Noah Giansiracusa investigated the connections 
between {\em tropicalization}, {\em $\mathbb{F}_1$-schemes} (i.e., schemes over the so-called 
{\em field with one element}), and {\em Berkovich analytification} of
algebraic varieties. A $\mathbb{K}$-scheme {\bf with a model over $\mathbb{F}_1$} is a scheme over a field $\mathbb{K}$ which admits an affine covering where the coordinate ring of each affine piece is a monoid ring and their gluing is induced by the localizations of monoids. In the affine case $X=\mathrm{Spec} (A)$, where $A$ is a 
$\mathbb{K}$-algebra and $\mathbb{K}$ is a valued field, the underlying set of the 
{\bf Berkovich analytification} $X^{\mathrm{an}}$   \index{Berkovich analytification} 
of $X$ is the set of all semivaluations of $A$ extending the given valuation on $\mathbb{K}$. In particular, $\mathbb{K}$ can be trivially valued, that is, $0\in \mathbb{K}$ has value $+\infty$ and all the other elements of $\mathbb{K}$ have value $0$. 
The pioneering work on this subject was the 2009 paper \cite{P 09} in which Sam Payne 
considered all possible embeddings of an affine algebraic variety $X$ 
into different affine spaces $\mathbb{K}^n$, viewed as toric varieties.  
He noted that a monomial map between such embeddings induces a natural 
map between the corresponding
extended global tropicalizations, and he showed that the associated category-theoretic limit 
of all such tropicalizations seen as sets can be identified with the analytification $X^{\mathrm{an}}$. 

In \cite{GG 16}, J. and N. Giansiracusa generalized  the construction of tropicalization from embeddings 
of some scheme $X$ over $\mathbb{K}$ to toric varieties to the
embeddings of $X$ to a $\mathbb{K}$-scheme equipped with a model over $\mathbb{F}_1$. 
Also, they endowed this tropicalization with a kind of scheme structure. 
In \cite{GG 22}, they constructed an embedding of a $\mathbb{K}$-scheme
$X$ to an $\mathbb{F}_1$-scheme $\widehat{X}$ which is \emph{universal} in the sense that its 
tropicalization admits a morphism to the tropicalization of $X$ with respect to any other embedding.
Moreover, they proved that the Berkovich analytification $X^{\mathrm{an}}$ can be recovered as the
tropicalization of this universal embedding $X\hookrightarrow \widehat{X}$. 
Finally, they reproved and refined Payne's limit result.

The paper \cite{L 23} of Oliver Lorscheid is an attempt to determine even a deeper structure 
than that of an $\mathbb{F}_1$-scheme that underlies all the instances of tropicalization. 
The author proposed for that role  \index{Ordered blueprints} 
his theory of \emph{ordered blueprints}. Also this theory enhances the tropicalization with a scheme
structure. One may find in \cite{L 23} the technical definitions of the ordered blueprint 
and the \emph{ordered blue scheme} as well as the references to other works of Lorscheid  
on the theory of blueprints.
Here we only draw the reader's attention to the fact that Lorscheid's work succeeds in interpreting 
Berkovich analytification, Kajiwara-Payne tropicalization \cite{K 08}, J. and N. Giansiracusa tropicalization,
Ulirsch tropicalization and some other approaches to tropical geometry via his theory
of ordered blueprints.

{\bf Acknowledgements: } 
   This research was funded, in whole or in part, by l'Agence Nationale de la Recherche (ANR), project SINTROP (ANR-22-CE40-0014) and Labex CEMPI (ANR-11-LABX-0007-01).  We thank Jose Luis Cisneros Molina,  D\~{u}ng Tr\'ang L\^e and Jose Seade for the invitation to contribute to this volume of the Handbook of Geometry and Topology of Singularities. We are grateful to Mar\'{\i}a Ang\'elica Cueto for everything we understood about local tropicalization during our collaboration with her and for her careful reading of the first version of this paper. We thank also Evelia Garc\'{\i}a Barroso, Arthur Renaudineau, Bernard Teissier and the anonymous referee for their remarks on previous versions of 
this text.

\noindent
\textbf{\small{Authors' addresses:}}
\smallskip
\

\noindent
\small{P.\ Popescu-Pampu,
  Univ.~Lille, CNRS, UMR 8524 - Laboratoire Paul Painlev{\'e}, F-59000 Lille, France.
  \\
\noindent \emph{Email address:} \url{patrick.popescu-pampu@univ-lille.fr}}
\vspace{2ex}

\noindent
\small{D.\ Stepanov, Department of Higher Mathematics and Center of Fundamental Mathematics, Moscow Institute of Physics and Technology,
9 Institutskiy per., Dolgoprudny, Moscow, 141701, Russia.
  \\
\noindent \emph{Email address:} \url{stepanov.da@phystech.edu}}

\medskip
\end{document}